\documentclass[11pt]{amsart}
\usepackage[T1]{fontenc}
\usepackage[latin9]{inputenc}
\usepackage{amsthm}
\usepackage{amssymb}
\usepackage{graphicx}
\makeatletter
\theoremstyle{plain}
\newtheorem{thm}{\protect\theoremname}[section]
  \theoremstyle{plain}
  \newtheorem{prop}[thm]{\protect\propositionname}
  \theoremstyle{plain}
  \newtheorem{cor}[thm]{\protect\corollaryname}
  \theoremstyle{plain}
  \newtheorem{lem}[thm]{\protect\lemmaname}
  \theoremstyle{definition}
  \newtheorem{defn}[thm]{\protect\definitionname}
 \theoremstyle{plain}
  \newtheorem{conj}[thm]{\protect\conjecturename}

\makeatother

  \providecommand{\corollaryname}{Corollary}
  \providecommand{\definitionname}{Definition}
  \providecommand{\lemmaname}{Lemma}
  \providecommand{\propositionname}{Proposition}
\providecommand{\theoremname}{Theorem}
\providecommand{\conjecturename}{Conjecture}

\usepackage{tikz}
\usetikzlibrary{decorations.markings, decorations.pathreplacing}
\tikzset{->-/.style={decoration={
  markings,
  mark=at position .5 with {\arrow{>}}},postaction={decorate}}}
  
\tikzset{->>-/.style={decoration={
  markings,
  mark=at position .5 with {\arrow{>>}}},postaction={decorate}}}
  
\tikzset{-<-/.style={decoration={
  markings,
  mark=at position .5 with {\arrow{<}}},postaction={decorate}}}
  
\tikzset{-<<-/.style={decoration={
  markings,
  mark=at position .5 with {\arrow{<<}}},postaction={decorate}}}

\begin{document}
\def\COMMENT#1{}

\global\long\def\labelenumi{(\roman{enumi})}

\def\noproof{{\unskip\nobreak\hfill\penalty50\hskip2em\hbox{}\nobreak\hfill%
        $\square$\parfillskip=0pt\finalhyphendemerits=0\par}\goodbreak}
\def\endproof{\noproof\bigskip}
\newdimen\margin   
\def\textno#1&#2\par{%
    \margin=\hsize
    \advance\margin by -4\parindent
           \setbox1=\hbox{\sl#1}%
    \ifdim\wd1 < \margin
       $$\box1\eqno#2$$%
    \else
       \bigbreak
       \hbox to \hsize{\indent$\vcenter{\advance\hsize by -3\parindent
       \sl\noindent#1}\hfil#2$}%
       \bigbreak
    \fi}
\def\proof{\removelastskip\penalty55\medskip\noindent{\bf Proof. }}

\newcommand{\Int}{{\rm Int}}

\title[Proof of a conjecture of Thomassen]{Proof of a conjecture of Thomassen on Hamilton cycles in highly connected tournaments}

\author{Daniela K\"uhn, John Lapinskas, Deryk Osthus and Viresh Patel}
\begin{abstract}
A conjecture of Thomassen from 1982 states that for every $k$ there is an $f(k)$ so that 
every strongly $f(k)$-connected tournament contains $k$ edge-disjoint Hamilton cycles.
A classical theorem of Camion, that every 
strongly connected tournament contains a Hamilton cycle, implies that $f(1)=1$.
So far, even the existence of $f(2)$ was open.
In this paper, we prove Thomassen's conjecture by showing that $f(k)=O( k^2 \log^2k)$.
This is best possible up to the logarithmic factor.
As a tool, we show that every strongly $10^4 k \log k$-connected tournament is $k$-linked
(which improves a previous exponential bound).
The proof of the latter is based on a fundamental result of Ajtai, Koml\'os and Szemer\'edi
on asymptotically optimal sorting networks.
\end{abstract}

\date{\today}

\thanks{D.~K\"uhn, D.~Osthus and V.~Patel were supported by the EPSRC, grant no.~EP/J008087/1. 
D.~K\"uhn was also supported by the ERC, grant no.~258345.
D.~Osthus was also supported by the ERC, grant no.~306349.} 

\maketitle

\section{Introduction\label{sec:intro}}

\subsection{Main result}
A \emph{tournament} is an orientation of a complete graph and a Hamilton cycle in a tournament is 
a (consistently oriented) cycle which contains all the vertices of the tournament.
Hamilton cycles in tournaments have a long and rich history.
For instance, one of the most basic results about tournaments is Camion's theorem, which states that every 
strongly connected tournament has a Hamilton cycle~\cite{camion}.
This is strengthened by Moon's theorem~\cite{moon}, which implies that such a tournament is even pancyclic,
i.e.~contains cycles of all possible lengths.
Many related results have been proved; the monograph by Bang-Jensen and Gutin~\cite{digraphsbook} gives an overview which also 
includes many recent results.

In 1982, Thomassen~\cite{thomassenconj} made a very natural conjecture on how to guarantee not just one Hamilton cycle, 
but many edge-disjoint ones: he conjectured that for every $k$ there is an $f(k)$ so that 
every strongly $f(k)$-connected tournament contains $k$ edge-disjoint Hamilton cycles
(see also the recent surveys~\cite{JBconj,KOsurvey}).
This turned out to be surprisingly difficult: not even the existence of $f(2)$ was known so far.
Our main result shows that $f(k)=O(k^{2}\log^{2} k)$.
\begin{thm}
\label{thm:thomconj}
There exists $C>0$ such that for all $k\in\mathbb{N}$ with $k\ge 2$ every
strongly $Ck^{2}\log^{2} k$-connected tournament contains $k$ edge-disjoint
Hamilton cycles.%
   \COMMENT{Need that $k\ge 2$ since $\log 1=0$.}
\end{thm}
In Proposition~\ref{thm:example}, we describe an example which shows that $f(k)\ge (k-1)^2/4$, 
i.e.~our bound on the connectivity is asymptotically close to best possible.
Thomassen~\cite{thomassenconj} observed that $f(2)>2$ and conjectured that $f(2)=3$.
He also observed that one cannot weaken the assumption in Theorem~\ref{thm:thomconj} by replacing 
strong connectivity with strong edge-connectivity.

To simplify the presentation, we have made no attempt to optimize the value of the constant~$C$.
Our exposition shows that one can take $C:=10^{12}$ for $k \ge 20$.\COMMENT{Lemmas~\ref{lem:good-embed} and \ref{lem:good}
tell us that any $10^7k\log k$-linked tournament has $k$ EDHCs for $k \ge 20$. Theorem \ref{thm:linkedness} tells us that any
$10^{12}k\log^2 k$-connected tournament is $10^7k\log k$-linked, since 
$10^4 \cdot (10^7 k \log k) \log (10^7 k \log k) \le 10^4 \cdot (10^7 k \log k) \log (k^9) \le 10^{12} k \log^2 k$
}
Rather than proving Theorem~\ref{thm:thomconj} directly, we deduce it as an immediate consequence of two
further results, which 
are both of independent interest: we show that every sufficiently highly connected tournament is highly linked
(see Theorem~\ref{thm:linkedness}) and show that every highly linked tournament contains many 
edge-disjoint Hamilton cycles (see Theorem~\ref{thm:mainresult}).

\subsection{Linkedness in tournaments}
Given sets $A$, $B$ of size $k$ in a strongly $k$-connected digraph~$D$, Menger's theorem implies that $D$
contains $k$ vertex-disjoint paths from $A$ to $B$. In a $k$-linked digraph, we can even specify 
the initial and final vertex of each such path (see Section~\ref{sec:notation} for the precise definition).
\begin{thm}
\label{thm:mainresult} There exists $C'>0$ such that for all $k\in\mathbb{N}$ with $k\ge 2$
every $C'k^{2}\log k$-linked tournament contains $k$ edge-disjoint
Hamilton cycles.%
\end{thm}
The bound in
Theorem~\ref{thm:mainresult} is asymptotically close to best possible, as we shall discuss below.
We will show that $C':=10^{7}$ works for all $k\ge 20$.
(As mentioned earlier, we have made no attempt to optimise the value of this constant.)

It is not clear from the definition that every (very) highly connected tournament is also highly linked.
In fact, for general digraphs this is far from true:
Thomassen~\cite{Thomassen} showed that for all $k$ there are strongly $k$-connected digraphs which
are not even $2$-linked. On the other hand, he showed that there is an (exponential) function $g(k)$
so that every strongly $g(k)$-connected tournament is $k$-linked~\cite{thomassenlinked}.
The next result shows that we can take $g(k)$ to be almost linear in $k$.
Note that this result together with Proposition~\ref{thm:example} shows that Theorem~\ref{thm:mainresult}
is asymptotically best possible up to logarithmic terms.
\begin{thm} \label{thm:linkedness}
For all $k \in \mathbb{N}$ with $k\ge 2$ every strongly $10^4k \log k$-connected tournament is $k$-linked. 
\end{thm}
For small~$k$, the constant $10^4$ can easily be improved (see Theorem~\ref{thm:small}).
The proof of Theorem~\ref{thm:linkedness} is based on a fundamental result of Ajtai, Koml\'os and Szemer\'edi~\cite{AKS1,AKS2}
on the existence of asymptotically optimal sorting networks. Though their result is asymptotically optimal, 
it is not clear whether this is the case for Theorem~\ref{thm:linkedness}. In fact, 
for the case of (undirected) graphs, a deep result of Bollob\'as and Thomason~\cite{BT96} states that
every $22k$-connected graph is $k$-linked (this was improved to~$10k$ by Thomas and Wollan~\cite{ThomasWollan}).
Thus one might believe that a similar relation also holds in the case of tournaments:
\begin{conj} \label{conj:linkage}
There exists $C>0$ such that for all $k \in \mathbb{N}$ every strongly $Ck$-connected tournament is $k$-linked. 
\end{conj}
Similarly, we believe that the logarithmic terms can also be removed in Theorems~\ref{thm:thomconj} and~\ref{thm:mainresult}:
\begin{conj} \label{conj:cycles} $ $
\begin{enumerate}
\item  There exists $C'>0$ such that for all $k \in \mathbb{N}$ every $C'k^2$-linked tournament contains $k$
edge-disjoint Hamilton cycles.
\item There exists $C''>0$ such that for all $k \in \mathbb{N}$ every strongly $C''k^2$-connected tournament contains $k$
edge-disjoint Hamilton cycles.
\end{enumerate} 
\end{conj}
Note that Conjectures~\ref{conj:linkage} and~\ref{conj:cycles}(i) together imply Conjecture~\ref{conj:cycles}(ii).

\subsection{Algorithmic aspects}
Both Hamiltonicity and linkedness in tournaments have also been studied from an algorithmic perspective.
Camion's theorem implies that the Hamilton cycle problem (though NP-complete in general) is solvable in polynomial time for tournaments.
Chudnovsky, Scott and Seymour~\cite{CSS} solved a long-standing problem of Bang-Jensen and Thomassen~\cite{BJT} by showing that the linkedness problem is 
also solvable in polynomial time for tournaments. More precisely, for a given tournament on $n$ vertices, one can determine in time
polynomial in $n$ whether it is $k$-linked and if yes, one can produce a corresponding set of $k$ paths (also in polynomial time).
Fortune, Hopcroft and Wyllie~\cite{linknp} showed that for general digraphs, the problem is NP-complete even for $k=2$.%
\COMMENT{check original paper VP: Lemma 3 of that paper gives NP-completeness of 2-linkage problem} 
We can use the result in~\cite{CSS} to obtain an algorithmic version of Theorem~\ref{thm:mainresult}. 
More precisely, given
a $C'k^2 \log k$-linked tournament on $n$ vertices, one can find $k$ edge-disjoint
Hamilton cycles in time polynomial in $n$ (where $k$ is fixed).
We discuss this in more detail in Section~\ref{concluding}.%
\COMMENT{so we should say at the end of the paper that we use CSS when we apply linkage to get paths}
Note that this immediately results in an algorithmic version of Theorem~\ref{thm:thomconj}.

\subsection{Related results and spanning regular subgraphs}\label{relresults}
Proposition~\ref{thm:example} actually suggests that the `bottleneck' to finding $k$ edge-disjoint Hamilton cycles is 
the existence of a $k$-regular subdigraph: 
it states that if the connectivity of a tournament $T$ is significantly lower than in Theorem~\ref{thm:thomconj}, then
$T$ may not even contain a spanning $k$-regular subdigraph.
There are other results which exhibit this phenomenon:
if $T$ is itself regular, then Kelly's conjecture from 1968 states that $T$ itself has a Hamilton decomposition.
Kelly's conjecture was proved very recently (for large tournaments) by K\"uhn and Osthus~\cite{Kelly}.

Erd\H{o}s raised a `probabilistic' version of Kelly's conjecture:
for a tournament $T$, let $\delta^0(T)$ denote the minimum of the minimum out-degree and the minimum in-degree.
He conjectured that for almost all tournaments $T$, the maximum number of edge-disjoint Hamilton cycles in $T$ 
is exactly $\delta^0(T)$. In particular, this would imply that with high probability, $\delta^0(T)$ is also the degree of
a densest spanning regular subdigraph in a random tournament~$T$.
This conjecture of Erd\H{o}s was proved by K\"uhn and Osthus~\cite{KellyII}, based on the main result in~\cite{Kelly}.

It would be interesting to obtain further conditions which relate the degree of the densest spanning regular subdigraph of a tournament~$T$
to the number of edge-disjoint Hamilton cycles in~$T$.
For undirected graphs, one such conjecture was made in~\cite{KLOmindeg}: it states that for any graph~$G$ satisfying the conditions of 
Dirac's theorem, the number of edge-disjoint Hamilton cycles in $G$ is exactly half the degree of a densest spanning even-regular subgraph of~$G$.
An approximate version of this conjecture was proved by Ferber, Krivelevich and Sudakov~\cite{FKS}, see e.g.~\cite{KLOmindeg,KellyII} for some related results.

\subsection{Organization of the paper}
The methods used in the current paper are quite different from those used e.g.~in the papers mentioned in Section~\ref{relresults}.
A crucial ingredient is the construction of highly structured dominating sets (see Section~\ref{sec:proofsketch} for an informal description).
We believe that this approach will have further applications.

In the next section, we introduce the notation that will
be used for the remainder of the paper. In Section~\ref{sec:proofsketch}, 
we give an overview of the proof of Theorem~\ref{thm:mainresult}.
In Sections~\ref{sec:linkedness} and~\ref{sec:example}, we give the relatively
short proofs of Theorem~\ref{thm:linkedness} and Proposition~\ref{thm:example}.
In Section~\ref{sec:onecycle},
we show that given a `linked domination structure' (as introduced in the proof sketch), 
we can find a single Hamilton cycle (Lemma~\ref{lem:mainengine}). In Section~\ref{sec:manycycles},
we show that given several suitable linked domination structures, we can repeatedly
apply Lemma~\ref{lem:mainengine} to find $k$ edge-disjoint Hamilton
cycles. In Section~\ref{sec:linkagegood} we show that any
highly linked tournament contains such suitable linked domination structures.
Finally, Section~\ref{concluding} contains some concluding remarks.

\section{Notation\label{sec:notation}}

The digraphs considered in this paper do not have loops and we allow up to two edges between any pair of $x$, $y$ of distinct vertices,
at most one in each direction. A digraph is an \emph{oriented graph} if there is at most one edge between any pair $x$, $y$ of distinct vertices,
i.e.~if it does not contain a cycle of length two. 

Given a digraph $D$, we write $V(D)$ for its vertex set, $E(D)$ for its edge set, $e(D):=|E(D)|$ for the number of its edges and $|D|$ for its \emph{order}, i.e.~for
the number of its vertices. We write $H \subseteq D$ to mean that $H$ is a subdigraph of $D$, i.e.\ $V(H) \subseteq V(D)$ and $E(H) \subseteq E(D)$.%
	\COMMENT{VP: added last sentence because it's used in Section~\ref{sec:linkedness}}
 Given $X\subseteq V(D)$, we write $D-X$ for the digraph obtained from $D$ by deleting all vertices in~$X$,
and $D[X]$ for the subdigraph of $D$ induced by~$X$.
Given $F\subseteq E(D)$, we write $D- F$ for the digraph obtained from $D$ by deleting all edges in~$F$.
We write $V(F)$ for the set of all endvertices of edges in~$F$.
If $H$ is a subdigraph of $D$, we write $D- H$ for $D- E(H)$.

We write $xy$ for an edge directed from $x$ to $y$. Unless stated otherwise, when we refer to paths and cycles in digraphs, we mean
directed paths and cycles, i.e.~the edges on these paths and cycles are oriented consistently.
Given a path $P=x\dots y$ from $x$ to $y$ and a vertex $z$ outside $P$ which sends an edge to $x$, we write $zxP$ for the path obtained from
$P$ by appending the edge $zx$. The \emph{length} of a path or cycle is the number of its edges. 
We call the terminal vertex of a path $P$ the \emph{head} of $P$ and denote it by $h(P)$.
Similarly, we call the initial vertex of a path $P$ the \emph{tail} of $P$ and denote it by $t(P)$.
The \emph{interior} $\Int(P)$ of a path $P$ is the subpath obtained by deleting $t(P)$ and $h(P)$.
Thus $\Int(P)=\emptyset$ if $P$ has length at most one. Two paths $P$ and $P'$ are \emph{internally disjoint}
if $P\neq P'$ and $V(\Int(P))\cap V(\Int(P'))=\emptyset$.%
   \COMMENT{need to add that $P\neq P'$ since otherwise $xy$ and $xy$ would be internally disjoint}
A \emph{path system} $\mathcal{P}$ is a collection of vertex-disjoint paths. We write $V(\mathcal{P})$ for the
set of all vertices lying on paths in $\mathcal{P}$ and $E(\mathcal{P})$ for the
set of all edges lying on paths in $\mathcal{P}$. We write $h(\mathcal{P})$ for the set consisting of%
   \COMMENT{TO DO: do we ever use this if $\mathcal{P}$ is not a path system, ie if the paths in $\mathcal{P}$
are not disjoint? JL: Not before section 7, still checking.}
the heads of all paths in $\mathcal{P}$ and $t(\mathcal{P})$ for the set consisting of
the tails of all paths in $\mathcal{P}$. If $v\in V(\mathcal{P})$, we write $v^{+}$ and
$v^{-}$ for the successor and predecessor of $v$ on
the path in $\mathcal{P}$ containing~$v$. A path system $\mathcal{P}$ is a \emph{path cover} of a directed graph $D$ if
every path in $\mathcal{P}$ lies in~$D$ and together the paths in $\mathcal{P}$ cover all the vertices of $D$.
If $X\subseteq V(D)$ and $\mathcal{P}$ is a path cover of $D[X]$, we sometimes also say that
$\mathcal{P}$ is a path cover of $X$.

If $x$ is a vertex of a digraph $D$, then $N^+_D(x)$ denotes the \emph{out-neighbourhood} of $x$, i.e.~the
set of all those vertices $y$ for which $xy\in E(D)$. Similarly, $N^-_D(x)$ denotes the \emph{in-neighbourhood} of $x$, i.e.~the
set of all those vertices $y$ for which $yx\in E(D)$. We write $d^+_D(x):=|N^+_D(x)|$ for the \emph{out-degree} of $x$ and
$d^-_D(x):=|N^-_D(x)|$ for its \emph{in-degree}. We denote the \emph{minimum out-degree} of $D$ by $\delta^+(D):=\min \{d^+_D(x): x\in V(D)\}$
and the \emph{maximum out-degree} of $D$ by $\Delta^+(D):=\max \{d^+_D(x): x\in V(D)\}$. We define the
\emph{minimum in-degree} $\delta^-(D)$ and the \emph{maximum in-degree} $\Delta^-(D)$ similarly.
The \emph{minimum degree} of $D$ is defined by $\delta(D):=\min\{d^+_D(x)+d^-_D(x): x\in V(D)\}$
and its \emph{minimum semi-degree} by $\delta^0(D):=\min\{\delta^+(D), \delta^-(D)\}$.
Whenever $X,Y\subseteq V(D)$ are disjoint, we write $e_D(X)$ for the number of edges of $D$ having both endvertices in~$X$, and
$e_D(X,Y)$ for the number of edges of $D$ with tail in $X$ and head in $Y$. 
We write $N^+_D(X):=\bigcup_{x\in X} N^+_D(x)$ and define $N^-_D(X)$ similarly.
In all these definitions we often omit the subscript $D$ if the digraph $D$ is clear from the context.

A digraph $D$ is \emph{strongly connected} if for all $x,y \in V(D)$, there is a directed path in
$D$ from $x$ to $y$. Given $k\in\mathbb{N}$, we say a digraph is \emph{strongly $k$-connected} if $|D|>k$ and
for every $S \subseteq V(D)$ of size at most $k-1$, $D-S$ is strongly connected.
We say a digraph $D$ is \emph{$k$-linked}
if $|D|\ge2k$ and whenever $x_{1},\dots,x_{k},y_{1},\dots,y_{k}$
are $2k$ distinct vertices of $D$, there exist vertex-disjoint paths
$P_{1},\dots,P_{k}$ such that $P_{i}$ is a path from $x_{i}$ to $y_{i}$.

Given a digraph $D$ and sets $X,Y\subseteq V(D)$, we say that $X$ \emph{in-dominates} $Y$ if each vertex in~$Y$
is an in-neighbour of some vertex in~$X$. Similarly, we say that $X$ \emph{out-dominates} $Y$ if each vertex in~$Y$
is an out-neighbour of some vertex in~$X$.\COMMENT{Note for JL and VP: This is \bf{different} from the definition we used earlier.}

A tournament $T$ is \emph{transitive}
if there exists an ordering $v_1,\dots, v_n$ of its vertices such that $v_iv_j\in E(T)$
if and only if $i<j$. In this case, we often say that $v_1$ is the \emph{tail} of $T$ and
$v_n$ is the \emph{head} of $T$.

Given $k\in\mathbb{N}$, we write $[k]:=\{1,\dots,k\}$. We write $\log$ for the binary logarithm and $\log^2 n:=(\log n)^2$.

\section{Sketch of the proof of Theorem~\ref{thm:mainresult}\label{sec:proofsketch}}

In this section, we give an outline of the proof of Theorem~\ref{thm:mainresult}.
An important idea is the notion of a `covering edge'.
Given a small  (pre-determined) set~$S$ of vertices in a tournament $T$, this will mean that it will suffice to find a cycle 
covering all vertices of $T-S$. 
More precisely, let $T$ be a tournament, let $x\in V(T)$, and
suppose $C$ is a cycle in $T$ covering $T-x$. If $yz\in E(C)$
and $yx,xz\in E(T)$, then we can replace $yz$ by $yxz$ in $C$
to turn $C$ into a Hamilton cycle. We call $yz$ a \emph{covering
edge} for $x$. More generally, if $S\subseteq V(T)$ and $C$ is a cycle
in $T$ spanning $V(T)-S$ such that $C$ contains a covering edge for
each $x\in S$, then we can turn $C$ into a Hamilton cycle by using
all these covering edges. Note that this idea still works  if
$C$ covers some part of $S$. On the other hand, note that~$S$ needs to be fixed at the beginning
-- this is different than in the recently popularized `absorbing method'.

Another important tool will be the following consequence of the Gallai-Milgram theorem:
suppose that $G$ is an oriented graph on $n$ vertices with $\delta(G)\ge n-\ell$.
Then the vertices of $G$ can be covered with $\ell$ vertex-disjoint paths. 
We use this as follows: suppose we are given a highly linked tournament $T$ and have already
found  $i$ edge-disjoint Hamilton cycles in~$T$.
Then the Gallai-Milgram theorem implies that we can cover the
vertices of the remaining oriented graph by a set of  $2i$ vertex-disjoint paths.
Very roughly, the aim is to link together these paths using the high linkedness of the original tournament $T$.

To achieve this aim, we introduce and use the idea of `transitive dominating sets'. Here
a transitive out-dominating set $A_\ell$ has the following properties:
\begin{itemize}
\item $A_\ell$ out-dominates $V(T) \setminus A_\ell$, i.e.~every vertex of $V(T) \setminus A_\ell$ receives an edge from~$A_\ell$.
\item $A_\ell$ induces a transitive tournament in $T$.
\end{itemize} 
Transitive in-dominating sets $B_\ell$ are defined similarly.

Now suppose that we have already found $i$ edge-disjoint Hamilton cycles in a highly linked tournament~$T$.
Let $T'$ be the oriented subgraph of $T$ obtained by removing the edges of these
Hamilton cycles.
Suppose that we also have the following `linked dominating structure' in $T'$, which consists of:
\begin{itemize}
\item small disjoint transitive out-dominating sets $A_1,\dots,A_t$, where $t:=2i+1$;
\item small disjoint transitive in-dominating sets $B_1,\dots,B_t$;
\item a set of short vertex-disjoint paths $P_1,\dots,P_t$, where each $P_\ell$ is a path from the head $b_\ell$ of $B_\ell$
to the tail $a'_\ell$ of $A_\ell$.
\end{itemize}
Recall that the head of a transitive tournament is the vertex of out-degree zero and
the tail is defined analogously.
The paths $P_\ell$ are found at the outset of the proof,  
using the assumption that the original tournament~$T$ is highly linked.
(Note that $T'$ need not be highly linked.)

Let $A^*$ denote the union of the $A_i$ and let $B^*$ denote the union of the $B_i$.
Note that $\delta(T'-A^* \cup B^*) \ge n-1-2i=n-t$. 
So the Gallai-Milgram theorem implies that we can cover the vertices of $T'-A^* \cup B^*$ with $t$
vertex-disjoint paths $Q_1,\dots,Q_t$. Now we can link up successive paths using the above dominating sets
as follows. The final vertex of $Q_1$ sends an edge to some vertex $b$ in $B_2$ (since $B_2$ is in-dominating).
Either $b$ is equal to the head $b_2$ of $B_2$ or
there is an edge in $T'[B_2]$ from $b$ to  $b_2$ (since $T'[B_2]$ is a transitive tournament).
Now follow the path $P_2$ from $b_2$ to the tail $a'_2$ of $A_2$. Using the fact that $T'[A_2]$ is transitive 
and that $A_2$ is out-dominating, we can similarly find a path of length at most two  
from $a'_2$ to the initial vertex of $Q_2$.
Continuing in this way, we can link up all the paths $Q_\ell$ and $P_\ell$ into a single cycle $C$ which covers all vertices outside
$A^* \cup B^*$ (and some of the vertices inside $A^* \cup B^*$). The idea is illustrated in Figure~\ref{fig:sketchproof}.
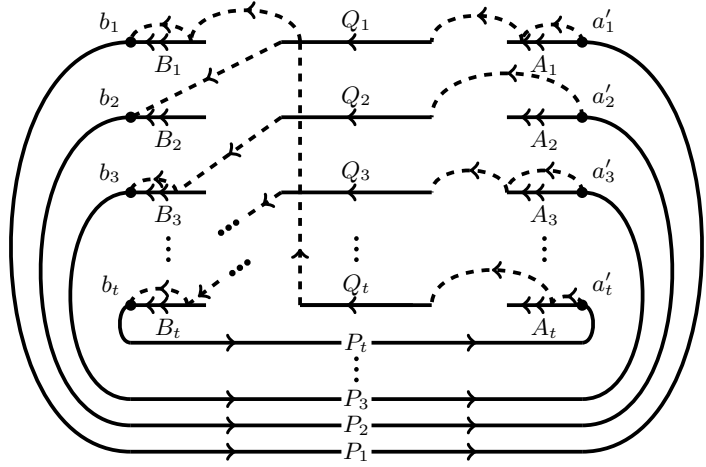
\begin{figure}
\centering\footnotesize
\begin{tikzpicture}[scale=1, line width = 0.5mm]
\tikzstyle{every node}=[font=\footnotesize]
\draw [-<<-] (0,0.5) -- (1,0.5); 
\draw [-<<-] (0,2) -- (1,2);
\draw [-<<-] (0,3) -- (1,3);
\draw [-<<-] (0,4) -- (1,4);

\node [draw=none, below] at (0.5, 0.45) {$B_t$};
\node [draw=none, below] at (0.5, 1.95) {$B_3$};
\node [draw=none, below] at (0.5, 2.95) {$B_2$};
\node [draw=none, below] at (0.5, 3.95) {$B_1$};

\draw [-<-] (2.25,0.5) -- (3.75,0.5); 
\draw (2.25, 0.5) -- (4, 0.5);
\draw [-<-] (2,2) -- (4,2);
\draw [-<-] (2,3) -- (4,3);
\draw [-<-] (2,4) -- (4,4);

\node [draw=none, above] at (3, 0.5) {$Q_t$};
\node [draw=none, above] at (3, 2) {$Q_3$};
\node [draw=none, above] at (3, 3) {$Q_2$};
\node [draw=none, above] at (3, 4) {$Q_1$};

\draw [-<<-] (5,0.5) -- (6,0.5); 
\draw [-<<-] (5,2) -- (6,2);
\draw [-<<-] (5,3) -- (6,3);
\draw [-<<-] (5,4) -- (6,4);

\node [draw=none, below] at (5.5, 0.45) {$A_t$};
\node [draw=none, below] at (5.5, 1.95) {$A_3$};
\node [draw=none, below] at (5.5, 2.95) {$A_2$};
\node [draw=none, below] at (5.5, 3.95) {$A_1$};

\draw (0,0.5) to [out=180, in=180] (0, 0); 
\draw (0,2) to [out=180, in=180] (0, -0.75);
\draw (0,3) to [out=180, in=180] (0, -1.1);
\draw (0,4) to [out=180, in=180] (0, -1.45);

\draw [->-] (0, 0) -- (2.8, 0);
\draw [->-] (0, -0.75) -- (2.8, -0.75);
\draw [->-] (0, -1.1) -- (2.8, -1.1);
\draw [->-] (0, -1.45) -- (2.8, -1.45);

\node[draw=none] at (3,0) {$P_t$};
\node[draw=none, font=\Huge] at (3, -0.25) {$\vdots$};
\node[draw=none] at (3,-0.75) {$P_3$};
\node[draw=none] at (3,-1.1) {$P_2$};
\node[draw=none] at (3,-1.45) {$P_1$};

\draw [->-] (3.2, 0) -- (6, 0);
\draw [->-] (3.2, -0.75) -- (6, -0.75);
\draw [->-] (3.2, -1.1) -- (6, -1.1);
\draw [->-] (3.2, -1.45) -- (6, -1.45);

\draw (6,0) to [out=0, in=0] (6, 0.5);
\draw (6,-0.75) to [out=0, in=0] (6, 2);
\draw (6,-1.1) to [out=0, in=0] (6, 3);
\draw (6,-1.45) to [out=0, in=0] (6, 4);

\node [draw=none, above left] at (-0, 4) {$b_1$}; 
\node [draw=none, above left] at (-0, 3) {$b_2$};
\node [draw=none, above left] at (-0, 2) {$b_3$};
\node [draw=none, above left] at (-0, 0.5) {$b_t$};

\node [draw=none, above right] at (6, 4) {$a_1'$}; 
\node [draw=none, above right] at (6, 3) {$a_2'$};
\node [draw=none, above right] at (6, 2) {$a_3'$};
\node [draw=none, above right] at (6, 0.5) {$a_t'$};

\node [circle, inner sep = 1.5pt, outer sep = 0pt, fill = black] at (0, 0.5) {};
\node [circle, inner sep = 1.5pt, outer sep = 0pt, fill = black] at (0, 2) {};
\node [circle, inner sep = 1.5pt, outer sep = 0pt, fill = black] at (0, 3) {};
\node [circle, inner sep = 1.5pt, outer sep = 0pt, fill = black] at (0, 4) {};
\node [circle, inner sep = 1.5pt, outer sep = 0pt, fill = black] at (6, 0.5) {};
\node [circle, inner sep = 1.5pt, outer sep = 0pt, fill = black] at (6, 2) {};
\node [circle, inner sep = 1.5pt, outer sep = 0pt, fill = black] at (6, 3) {};
\node [circle, inner sep = 1.5pt, outer sep = 0pt, fill = black] at (6, 4) {};

\node[draw=none, font=\Huge] at (0.5,1.35) {$\vdots$}; 
\node[draw=none, font=\Huge] at (3,1.35) {$\vdots$};
\node[draw=none, font=\Huge] at (5.5,1.35) {$\vdots$};

\draw [->-, dashed] (2, 4) -- (0, 3);
\draw [->-, dashed] (2, 3) -- (0.6, 2);
\draw [->-, dashed] (0.6, 2) to [out=90, in=90] (0, 2);
\draw [->-, dashed] (2, 2) -- (1.5, 1.66);
\path (2, 2) -- (1, 1.33) 
	node[pos=0.6, draw, circle, inner sep = 0.25pt, outer sep = 0pt, fill = black]{} 
	node[pos=0.7, draw, circle, inner sep = 0.25pt, outer sep = 0pt, fill = black]{} 
	node[pos=0.8, draw, circle, inner sep = 0.25pt, outer sep = 0pt, fill = black]{};
\draw [-<-, dashed] (0.75, 0.5) -- (1.25, 0.88);
\path (0.75, 0.5) -- (1.75, 1.21) 
	node[pos=0.6, draw, circle, inner sep = 0.25pt, outer sep = 0pt, fill = black]{} 
	node[pos=0.7, draw, circle, inner sep = 0.25pt, outer sep = 0pt, fill = black]{} 
	node[pos=0.8, draw, circle, inner sep = 0.25pt, outer sep = 0pt, fill = black]{};
\draw [->-, dashed] (0.75, 0.5) to [out=90, in=90] (0, 0.5);

\draw [->-, dashed] (2.25, 0.5) -- (2.25, 2);
\draw [dashed] (2.25, 2) -- (2.25, 4);
\draw [->-, dashed] (2.25, 4) to [out=90, in=90] (0.8, 4);
\draw [->-, dashed] (0.8, 4) to [out=90, in=90] (0, 4);

\draw [->-, dashed] (6, 4) to [out=90, in=90] (5.2, 4);
\draw [->-, dashed] (5.2, 4) to [out=90, in=90] (4, 4);
\draw [->-, dashed] (6, 3) to [out=90, in=90] (4, 3);
\draw [->-, dashed] (6, 2) to [out=90, in=90] (5, 2);
\draw [->-, dashed] (5, 2) to [out=90, in=90] (4,2);
\draw [->-, dashed] (6, 0.5) to [out=90, in=90] (5.6, 0.5);
\draw [->-, dashed] (5.6, 0.5) to [out=90, in=90] (4, 0.5);
\end{tikzpicture}

\caption{\footnotesize Illustrating the paths $Q_i$ and $P_i$ as well as the edges linking them up via the linked domination structure.}
\label{fig:sketchproof}
\end{figure}

In our construction, we will ensure that the paths $P_\ell$ contain a set of covering edges for $A^* \cup B^*$.
So $C$ also contains covering edges for $A^* \cup B^*$, and so we can transform $C$ into a Hamilton cycle as discussed earlier.

A major obstacle to the above strategy is that in order to guarantee the $P_\ell$ in $T'-A^* \cup B^*$, we would need
the linkedness of $T$ to be significantly larger than~$|A^* \cup B^*|$ (and thus larger than $|A_\ell|$).
However, there are many tournaments where any in- or out-dominating set contains $\Omega(\log n)$ vertices
(consider a random tournament).
This leads to a linkage requirement on $T$ which depends on $n$ (and not just on $k$, as required in Theorem~\ref{thm:mainresult}).

We overcome this problem by considering `almost dominating sets': instead of out-dominating all vertices outside $A_\ell$, the $A_\ell$
will out-dominate almost all vertices outside $A_\ell$. (Analogous comments apply to the in-dominating sets $B_\ell$.)
This means that we have a small `exceptional set' $E$ of vertices which are not out-dominated by all of the $A_\ell$.
The problem with allowing
an exceptional set is that if the tail of a path $Q_\ell$ in our cover is
in the exceptional set $E$, we cannot extend it directly into the out-dominating
set $A_\ell$ as in the above description. 
However, if we make sure that the $A_\ell$ include the vertices of smallest in-degree of~$T$, we can deal with this issue. 
Indeed, in this case we can show that every vertex $v \in E$ has in-degree $d^-(v)>2|E|$ say,
so we can always extend the tail of a path out of the exceptional set if necessary (and then into an almost out-dominating set $A_\ell$ as before). 
Unfortunately, we may `break' one
of the paths $P_\ell$\COMMENT{Forget to run this by you -- this used to be a $P_i$. -JL} in the process. However, if we are
careful about the place where we break it and construct some `spare' paths at the outset, it turns
out that the above strategy can be made to work.

\section{Connectivity and linkedness in tournaments}\label{sec:linkedness}
In this section we give the proof of Theorem~\ref{thm:linkedness}.
We will also collect some simple properties of highly linked directed graphs which we will use later on. 
The proof of Theorem~\ref{thm:linkedness} is based on an important result of Ajtai, Koml{\'o}s and Szemer{\'e}di~\cite{AKS1,AKS2}
on sorting networks. Roughly speaking, the proof idea of 
Theorem~\ref{thm:linkedness} is as follows. Suppose that we are given a highly connected tournament $T$
and we want to link an ordered set $X$ of $k$ vertices to a set $Y$ of the same size. Then we construct the equivalent of a sorting network $D$ inside $T-Y$
with `initial vertices' in $X$ and `final vertices' in a set $Z$. The high connectivity of $T$ guarantees an `unsorted' set of $k$ 
$ZY$-paths which avoid the vertices in $D-Z$. One can then extend these paths via 
$D$ to the appropriate vertices in $X$. In this way, we obtain paths linking the vertices in $X$ to the appropriate ones in $Y$. An example is shown in Figure~\ref{fig:linkfig}.
\begin{figure}
\centering\footnotesize
\includegraphics[scale=0.06]{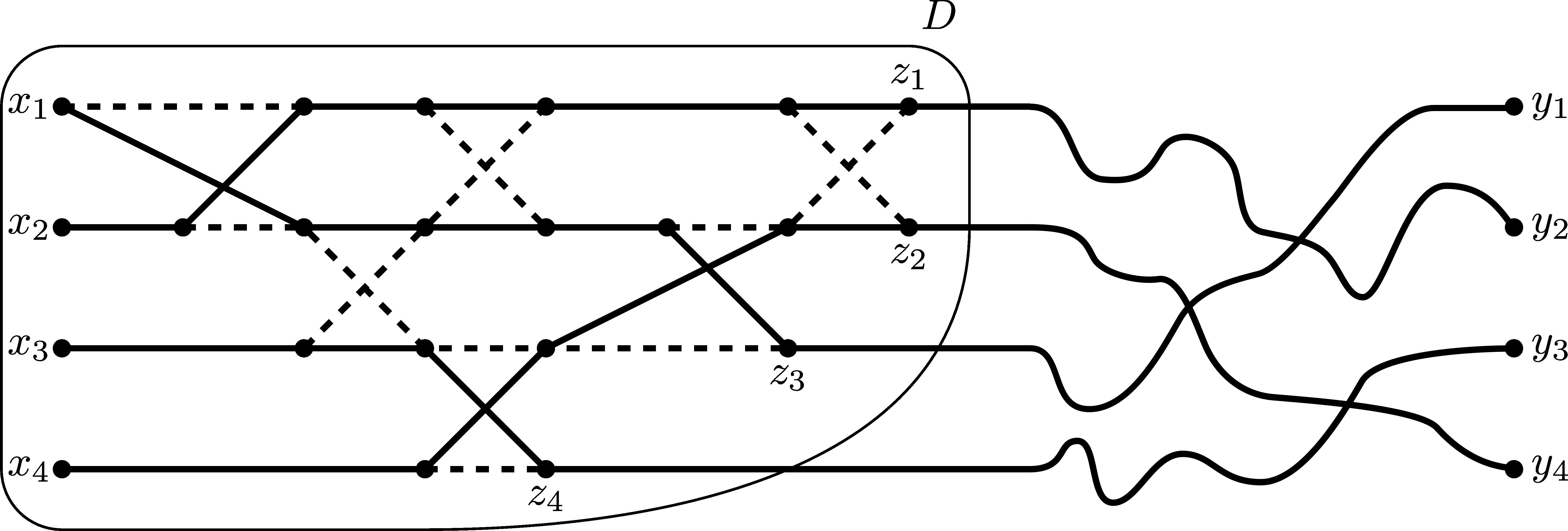}
\caption{\footnotesize Illustrating our construction of a digraph $D$ which corresponds to a sorting network for $k=4$.
$D$ is used to link $x_i$ to $y_i$. In the notation of the proof of Theorem~\ref{thm:linkedness}, we have $\pi(3)=1$.}
\label{fig:linkfig}
\end{figure}

We now introduce the necessary background on non-adaptive sorting algorithms and sorting networks; see~\cite{Knuth} for a more
detailed treatment.
In a sorting problem, we are given $k$ registers $R_1, \ldots, R_k$, and each register $R_i$ is assigned a distinct element from $[k]$,
which we call the \emph{value} of $R_i$; thus there is some permutation $\pi$ of $[k]$ such that value $i$ has been assigned to register $R_{\pi(i)}$.
Our task is to sort the values into their corresponding registers (so that value $i$ is assigned to $R_i$) by making a sequence of comparisons:
a comparison entails taking two registers and reassigning their values so that the higher value is assigned to the higher register
and the lower value to the lower register. A non-adaptive sorting algorithm is a sequence of comparisons specified in advance such
that for any initial assignment of $k$ values to $k$ registers, applying the prescribed sequence of comparisons results in every
value being assigned to its corresponding register. 

Ajtai, Koml{\'o}s and Szemer{\'e}di \cite{AKS1,AKS2} proved, via the construction of sorting networks, 
that there exists an absolute constant $C'$ and a non-adaptive sorting algorithm (for $k$ registers and values) that requires $C'k \log k$ comparisons,
and this is asymptotically best possible. It is known that we can take $C':=3050$ \cite{paterson}%
    \COMMENT{Our $log$ is binary - the one in Paterson's paper is binary too. So that's fine.}
(results of this type are often stated in terms of the \emph{depth} of a sorting network rather than the number of comparisons).

The next theorem is a consequence of the above. Before we can state it, we first need to introduce some notation.
A comparison $c$, which is part of some non-adaptive sorting algorithm for $k$ registers, will be denoted by $c=(s;t)$, where $1 \leq s < t \leq k$, to indicate
that $c$ is a comparison in which the values of registers $R_s$ and $R_t$ are compared (and sorted so the higher value
is assigned to the higher register).

\begin{thm}\label{lem:sorting}
Let $C':=3050$ and $k \in \mathbb{N}$ be such that $k\ge 2$. Then there exist $r\le C'k\log k$ and a sequence of comparisons $c_1, \ldots, c_r$ 
satisfying the following property: for any initial assignment of $k$ values to $k$ registers, 
applying the comparisons in sequence results in register $R_i$ being assigned the value $i$ for all $i\in [k]$.
\end{thm}

We now show how to obtain a structure within a highly connected tournament that simulates the function of a non-adaptive sorting algorithm.
Each comparison in the sorting algorithm 
will be simulated by a `switch', which we now define.
An \emph{$(a_1,a_2)$-switch} is a digraph $D$ on $5$ distinct vertices $a_1,a_2,b,b_1,b_2$, where either
$E(D) = \{a_1b, bb_1, bb_2, a_2b_1, a_2b_2\}$ or $E(D) = \{a_2b, bb_1, bb_2, a_1b_1, a_1b_2\}$. We call $b_1$ and $b_2$ the \emph{terminal vertices} of the $(a_1,a_2)$-switch.
Note that for any permutation $\pi$ of $\{1,2 \}$, there exist vertex-disjoint paths $P_1,P_2$ of $D$ such that $P_i$ joins $a_i$ to $b_{\pi(i)}$ for $i=1,2$. 

\begin{prop}\label{pr:switch}
Let $T$ be a tournament. Given distinct vertices $a_1,a_2 \in V(T)$, if $d^+_T(a_1), d^+_T(a_2) \geq 7$, then $T$ contains an $(a_1,a_2)$-switch.
\end{prop}
\begin{proof}
We may choose disjoint sets $A_1 \subseteq N^+_T(a_1) \setminus \{ a_2 \}$ and $A_2 \subseteq N^+_T(a_2) \setminus \{ a_1 \}$
with $|A_1|=|A_2|=3$. Consider the bipartite digraph $H$ induced by $T$ between $A_1$ and $A_2$. It is easy to check that
there exists $b \in A_1 \cup A_2$ with $d^+_H(b) \geq 2$. Let $b_1$ and $b_2$ be two out-neighbours of $b$ in $H$.
Now the vertices $a_1,a_2,b,b_1,b_2$ with suitably chosen edges from $T$ form an $(a_1,a_2)$-switch (with terminal vertices $b_1$ and $b_2$).
\end{proof}

Given $k \in \mathbb{N}$, we write $S_k$ for the set of permutations of $[k]$ and $id_k$ for the identity permutation of $[k]$.
The following structural lemma for tournaments is at the heart of the proof of Theorem~\ref{thm:linkedness}.
It constructs the equivalent of a sorting network in a tournament of high minimum outdegree.

\begin{lem}\label{lem:linkagestructure}
Let $C':=3050$ and $k \in \mathbb{N}$ be such that $k\ge 2$.
Let $T$ be a tournament with $\delta^+(T) \ge (3C'+5)k \log k$, and let $x_1, \ldots, x_k \in V(T)$ be distinct vertices. Then there exists a digraph $D \subseteq T$
 and distinct vertices $z_1, \ldots, z_k \in V(D)$ with the following properties:
\begin{enumerate}
\item $x_1, \ldots, x_k \in V(D)$.
\item $|D| \leq (3C'+1)k \log k$.
\item For any $\pi \in S_k$, we can find vertex-disjoint paths $P_1, \dots, P_k$ such that $P_i$ joins $x_{\pi(i)}$ to $z_{i}$ for all $i\in [k]$.
\end{enumerate}
\end{lem}
\begin{proof}
Consider the sorting problem for $k$ registers, and apply Theorem~\ref{lem:sorting} to obtain a sequence
$c_1, \ldots, c_r$ of $r \le C'k \log k$ comparisons such that for any $\pi \in S_k$, if value $i$ is
initially assigned to register $R_{\pi(i)}$, then applying the comparisons $c_1, \ldots, c_r$ results in every
value being assigned to its corresponding register. Given $\pi \in S_k$, we write $\pi_q \in S_k$ for the permutation
such that after applying the first $q$ comparisons $c_1,\dots,c_q$, value $i$ is assigned to register $R_{\pi_q(i)}$ for all $i$; 
thus $\pi_r = id_k$.

Let $D_0$ be the digraph with vertex set $\{x_1, \ldots, x_k \}$ and empty edge set.
We inductively construct digraphs $D_0 \subseteq D_1 \subseteq \dots \subseteq D_r \subseteq T$ and for each $D_q$
we maintain a set $Z_q = \{z_1^q, \ldots, z_k^q\}$ of $k$ distinct \emph{final vertices} such that the following holds:
\begin{itemize}
\item[(a)] $|D_q| = 3q + k$.
\item[(b)] Whenever $\pi \in S_k$ is a permutation,
there exist vertex-disjoint paths $P_1^q, \ldots, P_k^q$ in $D_q$ such that $P_i^q$ joins $x_{\pi(i)}$ to $z^q_{\pi_q(i)}$ for all $i\in [k]$.
\end{itemize}

Assuming the above statement holds for $q=0, \ldots, r$, then taking $D:=D_r$ with $z_i := z_i^r$ for all $i\in [k]$ proves the lemma.
Indeed $|D_r| = 3r + k \le 3C'k \log k + k \leq (3C'+1)k \log k$ and $\pi_r = id_k$.

Having already defined $D_0$, let us describe the inductive step of our construction. 
Suppose that for some $q\in [r]$ we have constructed $D_{q-1} \subseteq T$ and a corresponding set
$Z_{q-1} = \{ z_1^{q-1}, \ldots, z_k^{q-1} \}$ of final vertices.
Let $s,t\in [k]$ with $s<t$ be such that $c_q = (s;t)$. Define the tournament $T' := T - (V(D_{q-1}) \setminus \{z^{q-1}_s, z^{q-1}_t \} )$. 
Then $T'$ has minimum out-degree at least
\begin{align*}
(3C'+5)k \log k - |D_{q-1}| 
\geq (3C'+5)k\log k - 3r - k \geq 5k\log k - k \geq 7, 
\end{align*} 
and so in particular $d^+_{T'}(z_s^{q-1}), d^+_{T'}(z_t^{q-1}) \geq 7$. Thus we may apply Proposition~\ref{pr:switch} to obtain a
$(z_s^{q-1}, z_t^{q-1})$-switch $\sigma$ in $T'$. Write $b_1, b_2$ for the terminal vertices of $\sigma$.
Now $D_q$ is constructed from $D_{q-1}$ by adding the vertices and edges of $\sigma$ to $D_{q-1}$; note that
$z_s^{q-1}$ and $z_t^{q-1}$ are precisely the common vertices of $D_{q-1}$ and $\sigma$. 
We define the set $Z_q = \{z_1^q, \ldots, z_k^q \}$ by setting $z_i^q := z_i^{q-1}$  for all $i \neq s,t$ and $z_s^q := b_1$ 
as well as $z_t^q := b_2$.
Note that $z_1^q, \ldots, z_k^q $ are distinct.

Finally we check that conditions (a) and (b) hold for $D_q$. Condition (a) holds since $D_q$ has exactly $3$ more vertices than $D_{q-1}$.
For (b), 
by induction we may assume that there are vertex-disjoint paths $P^{q-1}_1, \ldots, P^{q-1}_k$ in $D_{q-1}$ such that $P^{q-1}_i$ joins $x_{\pi(i)}$
to $z_{\pi_{q-1}(i)}^{q-1}$ for all~$i\in [k]$. 
Choose vertex-disjoint paths $Q_s$ and $Q_t$ in $\sigma$ such that
\begin{itemize}
\item if $c_q$ swaps values in registers $R_s$ and $R_t$,
then $Q_s$ joins $z_s^{q-1}$ to $z_t^{q}$ and $Q_t$ joins $z_t^{q-1}$ to $z_s^{q}$;
\item if $c_q$ does not swap values in registers $R_s$ and $R_t$,
then $Q_s$ joins $z_s^{q-1}$ to $z_s^{q}$ and $Q_t$ joins $z_t^{q-1}$ to $z_t^{q}$.
\end{itemize}
Now exactly two of the paths from $P^{q-1}_1, \ldots, P^{q-1}_k$ end at $z_s^{q-1}$ and $z_t^{q-1}$,
namely those indexed by $\pi^{-1}_{q-1}(s)$ and $\pi^{-1}_{q-1}(t)$. We extend these two paths
using $Q_s$ and $Q_t$, and leave all others unchanged to obtain paths $P^q_1, \ldots, P^q_k$. 
It is straightforward to check that these paths are vertex-disjoint and that $P_i$ joins $x_{\pi(i)}$ to $z_{\pi_q(i)}^q$ for all $i\in [k]$.
\end{proof}

It is now an easy step to prove Theorem~\ref{thm:linkedness}.
We will use the following directed version of Menger's Theorem.  

\begin{thm}[Menger's Theorem] \label{th:Menger}
Suppose $D$ is a strongly $k$-connected digraph with $A,B \subseteq V(D)$ and $|A|, |B| \geq k$. Then there exist $k$ vertex-disjoint paths in $D$
each starting in $A$ and ending in $B$.%
\end{thm}

\removelastskip\penalty55\medskip\noindent{\bf Proof of Theorem~\ref{thm:linkedness}. }
Set $C':=3050$ and $C := 3C' + 6 < 10^4$.	
We must show that, given a strongly $Ck \log k$-connected tournament $T$
and distinct vertices $x_1, \ldots, x_k, y_1, \ldots, y_k \in V(T)$, we can find vertex-disjoint paths $R_1, \ldots, R_k$ such that $R_i$ joins $x_i$
to $y_i$ for all $i\in [k]$. 

Let $X:=\{ x_1, \ldots, x_k\}$, $Y:=\{ y_1, \ldots, y_k\}$ and $T':=T-Y$. Note that $T'$ is strongly $(3C' + 5)k\log k$-connected, and
in particular $\delta^+(T'') \ge (3C'+5)k\log k$.
Thus we can apply Lemma~\ref{lem:linkagestructure} to $T'$ and $x_1, \ldots, x_k$ to obtain a digraph $D \subseteq T'$ and
vertices $z_1, \ldots, z_k \in V(D)$ satisfying properties (i)--(iii) of Lemma~\ref{lem:linkagestructure}. 
Let $Z:=\{ z_1, \ldots, z_k\}$. Since $|D| \leq (3C'+1)k \log k$, the tournament $T'': = T - (V(D) \setminus Z)$ is strongly $k$-connected.
Therefore, by Theorem~\ref{th:Menger}, there exist $k$ vertex-disjoint paths, with each path starting in $Z$ and ending in $Y$.
For each $i\in [k]$, let us assume that $P_{\pi(i)}$ is the path that joins $z_{i}$ to $y_{\pi(i)}$, where $\pi$ is some permutation of $[k]$.
By Lemma~\ref{lem:linkagestructure}, we can find vertex-disjoint paths $Q_1, \ldots, Q_k$ in $D$ such that $Q_i$ joins $x_{\pi(i)}$ to $z_{i}$. 
Then the path $R_i:=Q_{\pi^{-1}(i)}P_{\pi^{-1}(i)}$ joins $x_i$ to $y_i$ and these paths are vertex-disjoint.
\endproof

Batcher~\cite{batcher} (see also~\cite{Knuth}) gave a construction of sorting networks which is asymptotically not optimal but which gives better 
values for small~$k$. More precisely, it uses at most $2k\log ^2 k$ comparisons for $k\ge 3$.%
    \COMMENT{The bound claimed in the paper is $(p^2-p+4)2^{p-2}-1$ for $k=2^p$. If $2^{p-1} \leq k \leq 2^p$, then $p-1 \leq \log k$, 
so $(p^2-p+4)2^{p-2}-1 \leq [p(p-1) + 4]2^{p-2} \leq [(\log k +1)\log k + 4]\frac{1}{2}k \leq 2k \log^2k $, where we use that $k\ge 3$}
If we use these as a building block in the proof of Lemma~\ref{lem:linkagestructure} 
instead of the asymptotically optimal ones leading to 
Theorem~\ref{lem:sorting}, we immediately obtain the following result which improves Theorem~\ref{thm:linkedness} 
for small values of~$k$.

\begin{thm} \label{thm:small}
For all $k \in \mathbb{N}$ with $k\ge 3$, every strongly  $12k \log^2 k$-connected 
tournament is $k$-linked. 
\end{thm}
For $k=2$, the best bound is obtained by a result of Bang-Jensen~\cite{JBconn}, who showed that every 
strongly $5$-connected semi-complete digraph is $2$-linked.%
    \COMMENT{Corollary 4.6. (every tournament is a quasitransitive digraph)} 

We will now collect some simple properties of highly linked directed graphs which we will use later on. 
The first two follow straightforwardly from the definition of linkedness.
\begin{prop}
\label{prop:samelinkage}Let $k\in\mathbb{N}$. Then a digraph $D$
is $k$-linked if and only if $|D|\ge2k$ and whenever $(x_{1},y_{1}),\dots,(x_{k},y_{k})$
are ordered pairs of (not necessarily distinct) vertices of $D$,
there exist internally disjoint paths $P_{1},\dots,P_{k}$
such that $P_{i}$ joins $x_{i}$ to $y_{i}$.
\end{prop}
\begin{prop}
\label{prop:robustlinkage}Let $k,\ell\in\mathbb{N}$ with $\ell<k$,
and let $D$ be a $k$-linked digraph. Let $X\subseteq V(D)$ and $F\subseteq E(D)$ be such that
$|X|+2|F| \le2\ell$. Then $D-X-F$ is $(k-\ell)$-linked.
\end{prop}

The next lemma shows that in a sufficiently highly linked digraph we can link given pairs of vertices
by vertex-disjoint paths which together do not contain too many vertices. 

\begin{lem}
\label{lem:shortlinkage}Let $k,s\in\mathbb{N}$, and let $D$
be a $2ks$-linked digraph. Let $(x_{1},y_{1}),\dots,(x_{k},y_{k})$
be ordered pairs of (not necessarily distinct) vertices in $D$. Then there exist internally disjoint
paths $P_{1},\dots,P_{k}$ such that $P_{i}$ joins $x_{i}$
to $y_{i}$ for all $i\in [k]$ and $|P_{1}\cup\dots\cup P_{k}|\le|D|/s$.\end{lem}
\begin{proof}
By Proposition~\ref{prop:samelinkage} there exist internally disjoint
paths $P_{1}^{1},\dots,P_{k}^{2s}$ such that $P_{i}^{j}$ joins
$x_{i}$ to $y_{i}$ for all $i\in [k]$ and all $j\in [2s]$. 
For any $j$, the interiors of $P_{1}^{j},\dots,P_{k}^{j}$ contain at least 
$|P_{1}^{j} \cup \dots \cup P_{k}^{j}|-2k$ vertices.
So the disjointness of the paths implies that
there is a $j \in [2s]$ with $|P_{1}^{j} \cup \dots \cup P_{k}^{j}|-2k \le |D|/2s$.
The result now follows by setting $P_i:=P_{i}^{j}$ and noting that $2k \le |D|/2s$.
\end{proof}


\section{Nearly extremal example\label{sec:example}}

The aim of this section is to prove the following proposition, which shows that the bound on the connectivity in Theorem~\ref{thm:thomconj} is close to best possible.

\begin{prop}\label{thm:example}
Fix $n,k \in \mathbb{N}$ with $k \ge 2$ and $n>k^2+k+2$. There exists a strongly $\lfloor k^2/4 \rfloor$-connected tournament $T$ of order $n$ 
such that if $D \subseteq T$ is a spanning $r$-regular subdigraph, then $r \leq k$. In particular, $T$ contains at most $k$ edge-disjoint Hamilton cycles. 
\end{prop}

It is easy to see that the above tournament $T$ is also $\Omega(k^2)$-linked.
This shows that the bound in Theorem~\ref{thm:mainresult} has to be at least quadratic in~$k$.
\begin{proof}
Let $\ell\in \mathbb{N}$. We will first describe a tournament $T_{\ell}=(V_{\ell},E_{\ell})$ of order $2\ell+1$ which is strongly $\ell$-connected. We then use
$T_\ell$ as a building block to construct a tournament as desired in the proposition. 

Let $V_{\ell} := \{v_0, \ldots, v_{2\ell} \}$ and let $E_{\ell}$ consist of the edges $v_iv_{i+t}$  for all $i=0, \ldots, 2\ell$ and all $t\in [\ell]$,
where indices are understood to be modulo $2\ell+1$. One may think of $T_{\ell}$ as the tournament with vertices $v_0 ,\ldots, v_{2\ell}$ placed in order, clockwise,
around a circle, where the out-neighbours of each $v_i$ are the $\ell$ closest vertices to $v_i$ in the clockwise direction, and the in-neighbours 
are the $\ell$ closest vertices in the anticlockwise direction. Note that $T_{\ell}$ is regular. 
Note also that, for any distinct $x,y \in V_{\ell}$, we can find a path in $T_{\ell}$ from $x$ to $y$ by traversing vertices from $x$ to $y$ in clockwise order;
this remains true even if we delete any $\ell - 1$ vertices from $T_{\ell}$. 

Next we construct a tournament $T_{m,\ell} = (V_{m,\ell}, E_{m,\ell})$ as follows. We take $V_{m,\ell}$ to be the disjoint
union of sets $A_{\ell} := \{a_0, \ldots, a_{2\ell}\}$, $B_{\ell} := \{b_0, \ldots, b_{2\ell}\}$, and $C_m := 
\{c_1, \ldots, c_m\}$. The edges of $T_{m,\ell}$ are defined as follows: $T_{m,\ell}[A_{\ell}]$ and $T_{m,\ell}[B_{\ell}]$ 
are isomorphic to $T_{\ell}$ (with the natural labelling of vertices), and $T[C_m]$ is a transitive 
tournament which respects the given order of the vertices in $C_m$
(i.e.~$c_ic_j$ is an edge if and only if $i<j$). Each vertex in $A_{\ell}$ is an in-neighbour of all 
vertices in $C_m$, and each vertex in $B_{\ell}$ is an out-neighbour of all vertices in~$C_m$. 
Finally, a vertex $a_i \in A_{\ell}$ is an in-neighbour of a vertex $b_j \in B_{\ell}$
if and only if $i \not= j$. Note that $|T_{m,\ell}|=m+4\ell+2$.
\medskip

\noindent
\textbf{Claim~1.} \emph{The tournament $T_{m,\ell}$ is strongly $\ell$-connected.}

\smallskip

\noindent
To see that $T_{m,\ell}$ is  strongly $\ell$-connected, we check that if $S \subseteq V_{m, \ell}$ with $|S| \leq \ell - 1$,
then $T_{m,\ell} - S$ is  strongly connected. Write $A_{\ell}'$, $B_{\ell}'$ and $C_m'$ respectively for $A_{\ell} \setminus S$, $B_{\ell} \setminus S$,
and $C_m \setminus S$. Note that there is at least one edge of $T_{m,\ell} - S$ from $B_{\ell}'$ to $A_{\ell}'$, which we may assume by symmetry to be $b_0a_0$.
Ordering the vertices of $T_{m,\ell}$ as $a_0, \ldots, a_{2\ell}, c_1, \ldots, c_{m}, b_1, \ldots, b_{2\ell}, b_0$
and removing the vertices of $S$ from this ordering gives a Hamilton cycle in $T_{m,\ell} - S$. Thus $T_{m,\ell} - S$ must be  strongly connected.
This completes the proof of Claim~1.

\medskip

\noindent
\textbf{Claim~2.} \emph{Let $m,\ell \in \mathbb{N}$ be such that $m > \sqrt{4\ell}$. 
Then for every $r$-regular spanning subdigraph $D \subseteq T_{m,\ell}$ we have $r \leq \sqrt{4\ell}$.}

\smallskip

\noindent
Suppose for a contradiction that $D \subseteq T_{m,\ell}$ is an $r$-regular spanning subdigraph with
$r:= \lfloor \sqrt{4\ell} \rfloor +1 > \sqrt{4\ell}$. 
Since $D$ is regular, we have $e_D(A_{\ell}, \bar{A}_{\ell}) = e_D(\bar{A_{\ell}}, A_{\ell})$, where $\bar{A_{\ell}}:=V(D) \setminus A_{\ell}$.
Noting that $r \leq m$, consider the first $r$ vertices $c_1, \ldots, c_r$ of $C_m$. Since 
$N_D^-(c_i) \subseteq N^-_{T_{m,\ell}}(c_i) = A_{\ell} \cup \{c_1, \ldots, c_{i-1} \}$ and $|N_D^-(c_i)|=r$, we have
$|N_D^-(c_i) \cap A_{\ell}| \geq r -i +1$, so that $e_D(A_{\ell}, \{c_i\}) \geq r-i+1$. Thus 
\[
e_D(\bar{A}_{\ell},A_{\ell}) = e_D(A_{\ell}, \bar{A}_{\ell}) \geq e(A_{\ell}, \{ c_1, \ldots, c_r \}) \geq r+ \cdots + 1 = \binom{r+1}{2}.
\] 
But $e_D(\bar{A}_{\ell},A_{\ell}) \leq e_{T_{m,\ell}}(\bar{A}_{\ell},A_{\ell}) = 2\ell+1$, so $\binom{r+1}{2} \leq 2\ell+1$.
This is easily seen to contradict $r > \sqrt{4\ell}$ for all $\ell \in \mathbb{N}$. This completes the proof of Claim~2.

\medskip

\noindent
To prove the proposition, we set $\ell := \lfloor k^2/4 \rfloor$ and $m := n - 4\ell - 2$, and take $T$ to be $T_{m,\ell}$.
Thus $|T| = |T_{m,\ell}| = m + 4\ell + 2 = n$. By Claim~1, $T$ is strongly $\lfloor k^2/4 \rfloor$-connected.
 Since $n > k^2 + k + 2 \geq 4\ell + \sqrt{4\ell} +2$, we have $m > \sqrt{4\ell}$, so
Claim~2 implies that if $D \subseteq T = T_{m,\ell}$ is a spanning $r$-regular subdigraph, then $r \leq \sqrt{4\ell} \leq k$.
\end{proof}


\section{Finding a single Hamilton cycle in suitable oriented graphs}\label{sec:onecycle}

We first state two simple, well-known facts concerning the degree
sequences of tournaments.
\begin{prop}
\label{prop:largedegreevertex}Let $T$ be a tournament on $n$ vertices.
Then $T$ contains at least one vertex of in-degree at most $n/2$,
and at least one vertex of out-degree at most $n/2$.
\end{prop}
\begin{prop}
\label{prop:degseqbound}Let $T$ be a tournament on $n$ vertices
and let $d \geq 0$. Then $T$ has at most $2d+1$ vertices
of in-degree at most $d$, and at most $2d+1$ vertices of out-degree
at most $d$.
\end{prop}

We will also use the following well-known result due to Gallai and Milgram
(see for example~\cite{BondyMurty}). (The \emph{independence number} of a digraph $T$
is the maximal size of a set $X\subseteq V(T)$ such that $T[X]$ contains no edges.)

\begin{thm}\label{thm:GM}
Let $T$ be a digraph with independence number at most $k$.
Then $T$ has a path cover consisting of at most $k$ paths.
\end{thm}

The following corollary is an immediate consequence of Theorem~\ref{thm:GM}.

\begin{cor}
\label{cor:pathcover}Let $T$ be an oriented graph on $n$ vertices with $\delta(T)\ge n-k$.
Then $T$ has a path cover consisting of at most $k$ paths.
\end{cor}


Given a digraph $T$, we define a \emph{covering edge} for a vertex
$v$ to be an edge $xy$ of $T$ such that $xv,vy\in E(T)$. We call $xv$
and $vy$ the \emph{activating edges} of $xy$. Note that if $xy$
is a covering edge for $v$ and $C$ is a cycle in $T$ containing
$xy$ but not $v$, we can form a new cycle $C'$ with $V(C')=V(C)\cup\{v\}$
by replacing $xy$ with $xvy$ in $C$. We will see in Section~\ref{sec:linkagegood} 
that covering edges are easy to find in strongly $2$-connected tournaments.

Recall that, given a path system $\mathcal{P}$, we write $h(\mathcal{P})$ for the
set of heads of paths in $\mathcal{P}$ and $t(\mathcal{P})$ for
the set of tails of paths in $\mathcal{P}$. If $v\in V(\mathcal{P})$, we write $v^{+}$ and
$v^{-}$ respectively for the successor and predecessor of $v$ on
the path in $\mathcal{P}$ containing~$v$. 

The following lemma allows us to take a path cover $\mathcal{P}$
of a digraph and modify it into a path cover $\mathcal{P}'$ with no heads in some ``bad'' set $I$,
without adding any heads or tails in $I\cup J$ for some other ``bad''
set $J$. Moreover, we can do this without losing any edges in some ``good'' set $F\subseteq E(\mathcal{P})$, and without altering too many paths in $\mathcal{P}$.
In our applications, $F$ will consist of covering edges. We require that every vertex
in $I$ has high out-degree. 
\begin{lem}
\label{lem:pathextend-out}Let $T$ be a digraph. Let $I,J\subseteq V(T)$
be disjoint.
Let $\mathcal{P}=\mathcal{P}_{1}\dot\cup\mathcal{P}_{2}$ be a path cover
of $T$ satisfying $h(\mathcal{P}_{2})\cap I=\emptyset$. Let $F\subseteq E(\mathcal{P})$.
Suppose $d^{+}(v)>3(|I|+|J|)+2|F|$ for all $v\in I$. Then there
exists a path cover $\mathcal{P}'$ of $T$ satisfying the following
properties:
\begin{enumerate}
\item $h(\mathcal{P}')\cap I=\emptyset$.
\item $h(\mathcal{P}')\cap J=h(\mathcal{P})\cap J$.
\item $t(\mathcal{P}')\cap(I\cup J)=t(\mathcal{P})\cap(I\cup J)$.
\item $F\subseteq E(\mathcal{P}')$.
\item $|\mathcal{P}'|\le|\mathcal{P}|+|\mathcal{P}_{1}|$.%
\COMMENT{Unfortunately, we can't easily apply the proof of Gallai-Milgram here to avoid increasing
$|\mathcal{P}|$ -- if we do, we lose all our covering edges.}
\item $|\mathcal{P}'\cap\mathcal{P}_{2}|\ge|\mathcal{P}_{2}|-|\mathcal{P}_{1}|$.
\end{enumerate}
If in addition $d^{+}(v)>3(|I|+|J|)+2|F|+|V(\mathcal{P}_{2})|$ for
all $v\in I$, then we may strengthen (vi) to $\mathcal{P}_{2}\subseteq\mathcal{P}'$.\end{lem}
\begin{proof}
We will use the degree condition on the vertices in~$I$ in the hypothesis to repeatedly extend
paths with heads in $I$ out of $I$, breaking other paths in $\mathcal{P}$
as a result. We must ensure that we do not create new
paths with endpoints in $I\cup J$ in the process. Let $r:=|\mathcal{P}_{1}|$ and $\mathcal{P}^{0}:=\mathcal{P}$.
We shall find path covers $\mathcal{P}^{1},\dots,\mathcal{P}^{r}$
of $T$ such that the following properties hold for all $0\le i\le r$:
\begin{enumerate}
\item [(P1)]$|h(\mathcal{P}^{i})\cap I|\le r-i$.
\item [(P2)]$h(\mathcal{P}^{i})\cap J=h(\mathcal{P})\cap J$.
\item [(P3)]$t(\mathcal{P}^{i})\cap(I\cup J)=t(\mathcal{P})\cap(I\cup J)$.
\item [(P4)]$F\subseteq E(\mathcal{P}^{i})$.
\item [(P5)]$|\mathcal{P}^{i}|\le|\mathcal{P}|+i$.
\item [(P6)]$|\mathcal{P}^{i}\cap\mathcal{P}_{2}|\ge|\mathcal{P}_{2}|-i$.
\end{enumerate}
If this is possible, we may then take $\mathcal{P}':=\mathcal{P}^{r}$.

By hypothesis, $\mathcal{P}^{0}$ satisfies (P1)--(P6). 
So suppose we have found $\mathcal{P}^{0},\dots,\mathcal{P}^{i-1}$ for
some $i\in [r]$. We then form $\mathcal{P}^{i}$ as follows. If
$|h(\mathcal{P}^{i-1})\cap I|\le r-i$, we simply let $\mathcal{P}^{i}:=\mathcal{P}^{i-1}$.%
      \COMMENT{This case could arise if not every path in $\mathcal{P}_0$ has head in $I$.}
Otherwise, let $P\in\mathcal{P}^{i-1}$ be a path with head $v\in I$.
We will form $\mathcal{P}^{i}$ by extending the head $v$ of $P$ and
breaking the path in $\mathcal{P}^{i-1}$ which $P$ now intersects
into two subpaths. Define
\[
X:=\{x\in V(T):\{x^{+},x,x^{-}\}\cap(I\cup J)\ne\emptyset\}.
\]
We have
\[
d^{+}(v)>3(|I|+|J|)+2|F|\ge|X|+|V(F)|\ge|X\cup V(F)|,
\]
and so there exists $w\in N^{+}(v)\setminus(X\cup V(F))$. Let $Q$
be the path in $\mathcal{P}^{i-1}$ containing $w$ (note that we may have $Q=P$). Split $Q$ into (at most)
two paths and an isolated vertex by removing any of the edges $w^{-}w,ww^{+}$ that exist,
and let $\mathcal{P}^{*}$ be the set of paths obtained from $\mathcal{P}^{i-1}$ in this way. Let $P^{*}$
be the path in $\mathcal{P}^{*}$ containing~$v$. (Note that $P^{*}=P$
unless $w \in V(P)$.) We then form $\mathcal{P}^{i}$ by replacing $P^{*}$
by $P^{*}vw$ in~$\mathcal{P}^{*}$. 

First suppose $w\in \textnormal{Int}(Q)$. Then  $\mathcal{P}^{i}$ is a path cover of $T$ such that%
\begin{align*}
h(\mathcal{P}^{i}) =(h(\mathcal{P}^{i-1})\setminus\{v\})\cup\{w,w^{-}\} \ \ \ \ \ \text{and} \ \ \ \ \ 
t(\mathcal{P}^{i}) =t(\mathcal{P}^{i})\cup\{w^{+}\}.
\end{align*}
Since $w\notin X$, we have $w,w^{-}\notin I$ and hence
\[
|h(\mathcal{P}^{i})\cap I|=|h(\mathcal{P}^{i-1})\cap I|-1\le r-i.
\]
Thus (P1) holds. Similarly, 
\begin{align*}
h(\mathcal{P}^{i})\cap J & =h(\mathcal{P}^{i-1})\cap J=h(\mathcal{P})\cap J,\\
t(\mathcal{P}^{i})\cap(I\cup J) & =t(\mathcal{P}^{i-1})\cap(I\cup J)=t(\mathcal{P})\cap(I\cup J),
\end{align*}
and so (P2) and (P3) hold. By similar arguments, (P1)--(P3) also hold if $w$ is an endpoint of~$Q$. Since $w\notin V(F)$ and $F\subseteq E(\mathcal{P}^{i-1})$
we have $F\subseteq E(\mathcal{P}^{i})$ and (P4) holds. (P5) holds too since $|\mathcal{P}^{i}|\le |\mathcal{P}^{i-1}|+1$.
Finally, we have altered at most two paths in $\mathcal{P}^{i-1}$. One of
these had its head in $I$, so we have altered at most one path in $\mathcal{P}^{i-1}\cap\mathcal{P}_{2}$.
Thus (P6) holds. 

If in addition we have
\[
d^{+}(v)>3(|I|+|J|)+2|F|+|V(\mathcal{P}_{2})|,
\]
then we may use almost exactly the same argument to prove the strengthened
version of the result. Instead of choosing $w\in N^{+}(v)\setminus(X\cup V(F))$,
we may choose $w \in N^{+}(v)\setminus(X\cup V(F) \cup V(\mathcal{P}_{2}))$.%
	\COMMENT{VP: The last sentence replaces 
``We simply redefine $F'$ to be the union of $V(\mathcal{P}_{2})$ and
the set of all endvertices of edges in $F$.'', where $F'$ was undefined.} 
We also strengthen (P6) to the requirement that
$\mathcal{P}_{2}\subseteq\mathcal{P}^{i}$. The strengthened (P6) must hold in each step since we now have that $w\notin V(\mathcal{P}_{2})$.
\end{proof}

The following analogue of Lemma~\ref{lem:pathextend-out} for tails can be obtained by reversing the orientation of each edge of~$T$.

\begin{lem}
\label{lem:pathextend-in}Let $T$ be a digraph. Let $I,J\subseteq V(T)$ be disjoint.
Let $\mathcal{P}=\mathcal{P}_{1}\dot\cup\mathcal{P}_{2}$
be a path cover of $T$ satisfying $t(\mathcal{P}_{2})\cap I=\emptyset$.
Let $F\subseteq E(\mathcal{P})$. Suppose $d^{-}(v)>3(|I|+|J|)+2|F|$
for all $v\in I$. Then there exists a path cover $\mathcal{P}'$
of $T$ satisfying the following properties:
\begin{enumerate}
\item $t(\mathcal{P}')\cap I=\emptyset$.
\item $t(\mathcal{P}')\cap J=t(\mathcal{P})\cap J$.
\item $h(\mathcal{P}')\cap(I\cup J)=h(\mathcal{P})\cap(I\cup J)$.
\item $F\subseteq E(\mathcal{P}')$.
\item $|\mathcal{P}'|\le|\mathcal{P}|+|\mathcal{P}_{1}|$.
\item $|\mathcal{P}'\cap\mathcal{P}_{2}|\ge|\mathcal{P}_{2}|-|\mathcal{P}_{1}|$. 
\end{enumerate}
If in addition $d^{-}(v)>3(|I|+|J|)+2|F|+|V(\mathcal{P}_{2})|$ for
all $v\in I$, then we may strengthen (vi) to $\mathcal{P}_{2}\subseteq\mathcal{P}'$.
\end{lem}

The following lemma is the main building block of the proof of Theorem~\ref{thm:mainresult}.
It will be applied repeatedly to find the required edge-disjoint Hamilton cycles. 
Roughly speaking, the lemma guarantees a Hamilton cycle 
provided that we have well-chosen disjoint (almost) dominating sets $A_i$ and $B_i$ which are linked by short paths
containing covering edges for all vertices in these dominating sets. (This is the linked dominating structure described in
Sections~\ref{sec:intro} and \ref{sec:proofsketch}.) An additional assumption is that
we have not removed too many edges of our tournament $T$ already. In general, the statement and proof roughly follow the sketch in
Section~\ref{sec:proofsketch}, with the addition of a set $X\subseteq V(T)$. 

The role of $X$ is as follows.
The sets $A_i$ and $B_i$ in the lemma dominate only almost all vertices of $T$, so we have some small
exceptional sets $E_A$ and $E_B$ of vertices which are not dominated.
We will  use Lemmas~\ref{lem:pathextend-out} and \ref{lem:pathextend-in}
to extend a certain path system out of these exceptional sets $E_A$ and $E_B$. 
For this we need that the vertices in $E_A\cup E_B$ have relatively high in-
and out-degree. But $T$ may have vertices which do not satisfy this degree condition. 
When we apply Lemma~\ref{lem:mainengine},
these problematic vertices will be the elements of $X$.

\begin{lem}
\label{lem:mainengine}Let $C:=10^{6}$, $k\ge20$, $t:=164k$,
and $c:=\lceil \log 50t+1\rceil$.%
    \COMMENT{VP: when checking constants, use $c \leq \log_{2}k + 15$}
Suppose that $T$ is an oriented graph of order
$n$ satisfying $\delta(T)>n-4k$ and $\delta^{0}(T)\ge Ck^{2}$.
Suppose moreover that $T$ contains disjoint sets of vertices $A_{1},\dots,A_{t}$,
$B_{1},\dots,B_{t}$ and $X$, a matching $F$, and vertex-disjoint
paths $P_{1},\dots,P_{t}$ such that the following conditions hold, where $A^{*}:=A_{1}\cup\dots\cup A_{t}$
and $B^{*}:=B_{1}\cup\dots\cup B_{t}$:
\begin{enumerate}
\item $2\le |A_{i}|\le c$ for%
   \COMMENT{Added the lower bound. It implies that $A$ and $A^\prime$ (as defined in the proof) are disjoint.}
all $i\in [t]$. Moreover, $T[A_{i}]$ is a transitive tournament 
whose head has out-degree at least $n/3$ in $T$.
\item There exists a set $E_A\subseteq V(T)\setminus (A^{*}\cup B^{*})$,
such that each $A_{i}$ out-dominates $V(T)\setminus (A^{*}\cup B^{*}\cup E_A)$.
Moreover, $|E_A|\le d^{-}/40$, where $d^{-}:=\min \{d_T^-(v) : v\in E_A\setminus X\}$.
\item $2\le |B_{i}|\le c$ for all $i\in [t]$. Moreover, $T[B_{i}]$ is a transitive tournament whose
tail has in-degree at least $n/3$ in $T$.
\item There exists a set $E_B\subseteq V(T)\setminus (A^{*}\cup B^{*})$,
such that each $B_{i}$ in-dominates $V(T)\setminus (A^{*}\cup B^{*}\cup E_B)$.
Moreover, $|E_B|\le d^{+}/40$, where $d^{+}:=\min\{d_T^+(v): v\in E_B\setminus X\}$.
\item For all $i\in [t]$, $P_{i}$ is a path from the head of $T[B_{i}]$ to the
tail of $T[A_{i}]$ which is internally disjoint from $A^{*}\cup B^{*}$.
Moreover, $|P_{1}\cup\dots\cup P_{t}|\le n/20$.
\item $F\subseteq E(P_{1}\cup\dots\cup P_{t})$ and $V(F) \cap (A^*\cup B^*) = \emptyset$. Moreover, $F = \{e_v:v\in A^* \cup B^*\}$,
where $e_v$ is a covering edge for $v$ and $e_v \ne e_{v'}$ whenever $v \ne v'$. In particular, $|F|=|A^*\cup B^*|\le 2ct$.
\item We have $X\subseteq V(P_{1}\cup\dots\cup P_{t})$, $X\cap(A^{*}\cup B^{*})=\emptyset$
and $|X|\le2kt$.
\end{enumerate}
Then $T$ contains a Hamilton cycle.\end{lem}
\begin{proof}
Without loss of generality, suppose that $d^{-}\le d^{+}$. (Otherwise,
reverse the orientation of every edge in $T$.) Write $a_{i}$ for the head of $T[A_{i}]$ and
 $a_{i}'$ for its tail. Similarly, write $b_{i}$ for the head of $T[B_{i}]$ and $b_{i}'$ for its tail.
Let 
$$
A:=\{a_{1},\dots,a_{t}\}, \  \ A^\prime:=\{a_{1}',\dots,a_{t}'\}, \ \ B:=\{b_{1},\dots,b_{t}\} \ \ \mbox{and} \ \ B^\prime:=\{b_{1}',\dots,b_{t}'\}.
$$
Thus the sets $A,A^\prime,B,B^\prime$ are disjoint, and by condition~(v) the paths $P_i$ join $B$ to $A^\prime$. Let
\begin{align*}
N  :=V(T)\setminus(A^{*}\cup B^{*}), \ \ \ 
T'  :=T[N\cup A^\prime\cup B], \ \ \ \mbox{and} \ \ \ 
\mathcal{P}_{2} & :=\{P_{1},\dots,P_{t}\}.
\end{align*}
By Corollary~\ref{cor:pathcover}, there exists a path cover $\mathcal{P}_{1}$
of $N\setminus V(\mathcal{P}_{2})$ with $|\mathcal{P}_{1}|\le4k$.
Then $\mathcal{Q}_{1}:=\mathcal{P}_{1}\dot\cup\mathcal{P}_{2}$ is a path
cover of $T'$. The situation is illustrated in Figure~\ref{fig:mainfig}.
\begin{figure}
\centering
\begin{tikzpicture}[scale=1, line width = 0.5mm]
\tikzstyle{every node}=[font=\footnotesize]

\draw [-<<-] (0,0.5) -- (1,0.5); 
\draw [-<<-] (0,2) -- (1,2);
\draw [-<<-] (0,3) -- (1,3);
\draw [-<<-] (0,4) -- (1,4);

\node [draw=none, below] at (0.5, 0.45) {$B_t$};
\node [draw=none, below] at (0.5, 1.95) {$B_3$};
\node [draw=none, below] at (0.5, 2.95) {$B_2$};
\node [draw=none, below] at (0.5, 3.95) {$B_1$};

\draw [-<-] (2.25,0.5) -- (3.75,0.5); 
\draw [-<-] (2.25,2) -- (3.75,2);
\draw [-<-] (2.25,3) -- (3.75,3);
\draw [-<-] (2.25,4) -- (3.75,4);

\node [draw=none, above] at (3, 0.5) {$Q_t$};
\node [draw=none, above] at (3, 2) {$Q_3$};
\node [draw=none, above] at (3, 3) {$Q_2$};
\node [draw=none, above] at (3, 4) {$Q_1$};

\draw [-<<-] (5,0.5) -- (6,0.5); 
\draw [-<<-] (5,2) -- (6,2);
\draw [-<<-] (5,3) -- (6,3);
\draw [-<<-] (5,4) -- (6,4);

\node [draw=none, below] at (5.5, 0.45) {$A_t$};
\node [draw=none, below] at (5.5, 1.95) {$A_3$};
\node [draw=none, below] at (5.5, 2.95) {$A_2$};
\node [draw=none, below] at (5.5, 3.95) {$A_1$};

\draw [line width = 0.2mm] (-0.1, 0.9) -- (-0.1, 0.4) -- (0.1, 0.4) -- (0.1, 0.9);
\draw [line width = 0.2mm] (0.1, 1.6) -- (0.1, 4.1) -- (-0.1, 4.1) -- (-0.1, 1.6);

\draw [line width = 0.2mm] (0.9, 0.9) -- (0.9, 0.4) -- (1.1, 0.4) -- (1.1, 0.9);
\draw [line width = 0.2mm] (1.1, 1.6) -- (1.1, 4.1) -- (0.9, 4.1) -- (0.9, 1.6);

\draw [line width = 0.2mm] (4.9, 0.9) -- (4.9, 0.4) -- (5.1, 0.4) -- (5.1, 0.9);
\draw [line width = 0.2mm] (5.1, 1.6) -- (5.1, 4.1) -- (4.9, 4.1) -- (4.9, 1.6);

\draw [line width = 0.2mm] (5.9, 0.9) -- (5.9, 0.4) -- (6.1, 0.4) -- (6.1, 0.9);
\draw [line width = 0.2mm] (6.1, 1.6) -- (6.1, 4.1) -- (5.9, 4.1) -- (5.9, 1.6);

\node [draw=none, above] at (0,4.05) {$B$};
\node [draw=none, above] at (1,4.05) {$B'$};
\node [draw=none, above] at (5,4.05) {$A$};
\node [draw=none, above] at (6,4.05) {$A'$};

\draw [line width = 0.2mm, decorate, decoration={brace,amplitude=8pt}, yshift=13pt] (-0.2, 4.05) -- (1.2,4.05) node [midway, yshift=16pt] {$B^*$};
\draw [line width = 0.2mm, decorate, decoration={brace,amplitude=8pt}, yshift=13pt] (4.8, 4.05) -- (6.2,4.05) node [midway, yshift=16pt] {$A^*$};

\draw (0,0.5) to [out=180, in=180] (0, 0); 
\draw (0,2) to [out=180, in=180] (0, -0.75);
\draw (0,3) to [out=180, in=180] (0, -1.1);
\draw (0,4) to [out=180, in=180] (0, -1.45);

\draw [->-] (0, 0) -- (2.80, 0);
\draw [->-] (0, -0.75) -- (2.80, -0.75);
\draw [->-] (0, -1.1) -- (2.80, -1.1);
\draw [->-] (0, -1.45) -- (2.80, -1.45);

\node[draw=none] at (3,0) {$P_t$};
\node[draw=none, font=\Huge] at (3, -0.25) {$\vdots$};
\node[draw=none] at (3,-0.75) {$P_3$};
\node[draw=none] at (3,-1.1) {$P_2$};
\node[draw=none] at (3,-1.45) {$P_1$};

\draw [->-] (3.2, 0) -- (6, 0);
\draw [->-] (3.2, -0.75) -- (6, -0.75);
\draw [->-] (3.2, -1.1) -- (6, -1.1);
\draw [->-] (3.2, -1.45) -- (6, -1.45);

\draw (6,0) to [out=0, in=0] (6, 0.5);
\draw (6,-0.75) to [out=0, in=0] (6, 2);
\draw (6,-1.1) to [out=0, in=0] (6, 3);
\draw (6,-1.45) to [out=0, in=0] (6, 4);

\node [draw=none, above left] at (-0, 4) {$b_1$}; 
\node [draw=none, above left] at (-0, 3) {$b_2$};
\node [draw=none, above left] at (-0, 2) {$b_3$};
\node [draw=none, above left] at (-0, 0.5) {$b_t$};

\node [draw=none, right] at (1, 4) {$b_1'$}; 
\node [draw=none, right] at (1, 3) {$b_2'$}; 
\node [draw=none, right] at (1, 2) {$b_3'$}; 
\node [draw=none, right] at (1, 0.5) {$b_t'$}; 

\node [draw=none, left] at (5, 4) {$a_1$}; 
\node [draw=none, left] at (5, 3) {$a_2$};
\node [draw=none, left] at (5, 2) {$a_3$};
\node [draw=none, left] at (5, 0.5) {$a_t$};

\node [draw=none, above right] at (6, 4) {$a_1'$}; 
\node [draw=none, above right] at (6, 3) {$a_2'$};
\node [draw=none, above right] at (6, 2) {$a_3'$};
\node [draw=none, above right] at (6, 0.5) {$a_t'$};

\node[draw=none, font=\Huge] at (0.5,1.35) {$\vdots$}; 
\node[draw=none, font=\Huge] at (3,1.35) {$\vdots$};
\node[draw=none, font=\Huge] at (5.5,1.35) {$\vdots$};

\node [circle, inner sep = 1.5pt, outer sep = 0pt, fill = black] at (0, 0.5) {};
\node [circle, inner sep = 1.5pt, outer sep = 0pt, fill = black] at (0, 2) {};
\node [circle, inner sep = 1.5pt, outer sep = 0pt, fill = black] at (0, 3) {};
\node [circle, inner sep = 1.5pt, outer sep = 0pt, fill = black] at (0, 4) {};
\node [circle, inner sep = 1.5pt, outer sep = 0pt, fill = black] at (1, 0.5) {};
\node [circle, inner sep = 1.5pt, outer sep = 0pt, fill = black] at (1, 2) {};
\node [circle, inner sep = 1.5pt, outer sep = 0pt, fill = black] at (1, 3) {};
\node [circle, inner sep = 1.5pt, outer sep = 0pt, fill = black] at (1, 4) {};
\node [circle, inner sep = 1.5pt, outer sep = 0pt, fill = black] at (5, 0.5) {};
\node [circle, inner sep = 1.5pt, outer sep = 0pt, fill = black] at (5, 2) {};
\node [circle, inner sep = 1.5pt, outer sep = 0pt, fill = black] at (5, 3) {};
\node [circle, inner sep = 1.5pt, outer sep = 0pt, fill = black] at (5, 4) {};
\node [circle, inner sep = 1.5pt, outer sep = 0pt, fill = black] at (6, 0.5) {};
\node [circle, inner sep = 1.5pt, outer sep = 0pt, fill = black] at (6, 2) {};
\node [circle, inner sep = 1.5pt, outer sep = 0pt, fill = black] at (6, 3) {};
\node [circle, inner sep = 1.5pt, outer sep = 0pt, fill = black] at (6, 4) {};

\draw [dashed] (3.3, 1.6) -- (3.3, 4.3) to [out=90, in=0] (3.1, 4.5) -- (2, 4.5) to [out=180, in=90] (1.8, 4.3) -- (1.8, 2.6) to [out=270, in=180] (2, 2.4)
	-- (2.5, 2.4) to [out=0, in=90] (2.7, 2.2) -- (2.7, 1.6);
\draw [dashed] (2.7, 0.9) -- (2.7, 0.5) to [out=270, in=180] (2.9, 0.3) -- (3.1, 0.3) to [out=0, in=270] (3.3, 0.5) -- (3.3, 0.9);
\node [draw=none, above] at (2.55, 4.5) {$E_B$};

\draw [dotted] (4, 1.6) -- (4, 3.3) to [out=90, in=0] (3.8, 3.5) -- (2.2, 3.5) to [out=180, in=90] (2, 3.3) -- (2, 2.8) to [out=270, in=180] (2.2, 2.7)
	-- (3.3, 2.7) to [out=0, in=90] (3.5, 2.5) -- (3.5, 1.6);
\draw [dotted] (3.5, 0.9) -- (3.5, 0.5) to [out=270, in=180] (3.7, 0.3) -- (3.8, 0.3) to [out=0, in=270] (4, 0.5) -- (4, 0.9);
\node [draw=none, above right] at (3.8, 3.4) {$E_A$};

\end{tikzpicture}
\caption{\footnotesize Our linked domination structure and path cover at the beginning of the proof of Lemma~\ref{lem:good-embed}.}
\label{fig:mainfig}
\end{figure}
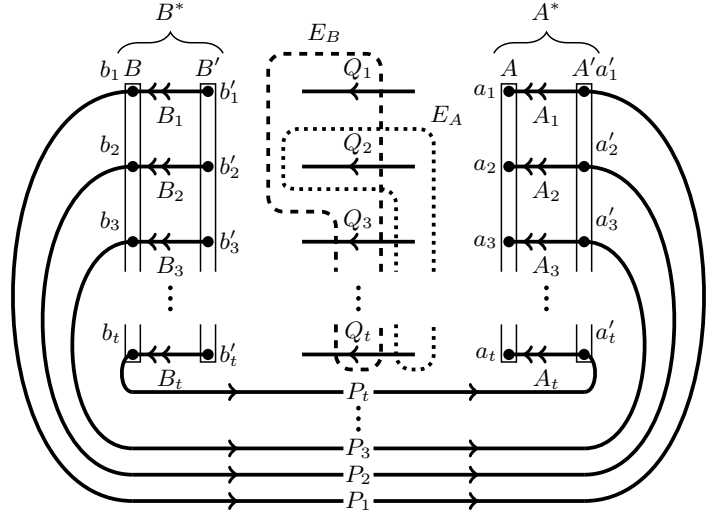

\medskip

\noindent
\textbf{Claim.} \emph{There exists an oriented graph $T''$ with $T'\subseteq T''\subseteq T[V(T')\cup A\cup B^\prime]$ and
a path cover $\mathcal{Q}$ of $T''$ such that the following properties hold:
\begin{enumerate}
\item [(Q1)]$F\subseteq E(\mathcal{Q})$.
\item [(Q2)]$t(\mathcal{Q})\cap E_A=\emptyset$.
\item [(Q3)]$h(\mathcal{Q})\cap E_B=\emptyset$.
\item [(Q4)]$|\mathcal{Q}\cap\mathcal{P}_{2}|\ge|\mathcal{Q}_{1}|-20k$.
\item [(Q5)]If $a_i$ or $b'_i$ is in $V(\mathcal{Q})$, then $P_i\notin \mathcal{Q}$.
\item [(Q6)]$|\mathcal{Q}|\le|\mathcal{Q}_{1}|+124k$.
\item [(Q7)]No paths in $\mathcal{Q}\setminus\mathcal{P}_{2}$ have endpoints
in $A^{*}\cup B^{*}$.\COMMENT{I've renumbered the old (Q7) to the new (Q5). -JL}
\end{enumerate}}

\smallskip

We will prove the claim by applying Lemmas~\ref{lem:pathextend-out} and \ref{lem:pathextend-in} repeatedly to improve
our current path cover. More precisely, we will construct path covers $\mathcal{Q}_2,\dots,\mathcal{Q}_6$ such that eventually
$\mathcal{Q}_6$ satisfies (Q1)--(Q7). So we can take $\mathcal{Q}:=\mathcal{Q}_6$.

In order to be able to apply Lemmas~\ref{lem:pathextend-out} and \ref{lem:pathextend-in}, we must first bound the degrees
of the vertices in $T'$ from below. For all $v\in V(T')$, we have
\begin{align}\label{eq:d+bound-1}
d_{T'}^{+}(v) \ge d_{T}^{+}(v)-|A^{*}\cup B^{*}|\stackrel{{\rm (i)}, {\rm (iii)}}{\ge} d_{T}^{+}(v)-2ct
\ge d_{T}^{+}(v)-\frac{\delta^{0}(T)}{5}
\ge \frac{4}{5}d_{T}^{+}(v).
\end{align}
Similarly,
\begin{align}
d_{T'}^{-}(v) & \ge\frac{4}{5}d_{T}^{-}(v)\label{eq:d-bound-1}
\end{align}
for all $v\in V(T')$.%
   \COMMENT{This argument needs pictures *really* badly! I'd suggest one for the construction of $\mathcal{Q}$
from $\mathcal{Q}_1$, and another for the construction of $C$ from $\mathcal{Q}$.
DO: one large one towards the beginning is probably enough...}

We will first extend the tails of paths in $\mathcal{Q}_1$ out of $E_A$.
We do this by applying Lemma~\ref{lem:pathextend-in} to $T'$ and $\mathcal{Q}_{1}=\mathcal{P}_{1}\dot\cup\mathcal{P}_{2}$
with $I:=E_A\setminus X$, $J:=X\cup A^\prime\cup B$ 
to form a new path cover $\mathcal{Q}_{2}$ of $T'$ which will satisfy~(Q1) and~(Q2).
By conditions (ii) and (v),  no paths in $\mathcal{P}_{2}$
have endpoints in $I$. By condition (vi),  $F\subseteq E(\mathcal{Q}_{1})$. 
Moreover, 
\begin{eqnarray}\label{eq:help}
3(|I|+|J|)+2|F| & \le & 3|E_A|+3|X|+3|A^\prime|+3|B|+2|F|\nonumber\\
 & \stackrel{{\rm (ii)}, {\rm (vii)}, {\rm (vi)}}{\le} & \frac{3}{40}d^{-}+6kt+6t+4ct < \frac{4}{5}d^-.
\end{eqnarray}
In the final inequality we used the fact that $d^{-}\ge\delta^{0}(T)\ge Ck^{2}$.
Thus for all $v\in I$ we have
\[
d_{T'}^{-}(v)\stackrel{(\ref{eq:d-bound-1})}{\ge} \frac{4}{5} d_{T}^{-}(v)\stackrel{\rm (ii)}{\ge} \frac{4}{5} d^-
\stackrel{(\ref{eq:help})}{>} 3(|I|+|J|)+2|F|.
\]
Thus the requirements of Lemma~\ref{lem:pathextend-in} are satisfied, and we can apply the lemma to obtain a path cover $\mathcal{Q}_{2}$
of $T'$. 

Lemma~\ref{lem:pathextend-in}(iv) implies that
$\mathcal{Q}_{2}$ satisfies (Q1). Moreover, Lemma~\ref{lem:pathextend-in}(v),(vi) imply that
\begin{eqnarray} \label{Q12eq}
|\mathcal{Q}_{2}|\le |\mathcal{Q}_{1}|+4k \ \ \ &\mbox{as well as }& \ \ \ |\mathcal{Q}_{2}\cap\mathcal{P}_{2}| \ge |\mathcal{P}_{2}|-4k \geq |\mathcal{Q}_{1}|-8k, \\
\ \ \ &\mbox { and thus }& \ \ \ |\mathcal{Q}_2 \setminus \mathcal{P}_2| \le 12k, \nonumber
\end{eqnarray}
where we have used that $|\mathcal{Q}_{1}| = |\mathcal{P}_{1}| + |\mathcal{P}_{2}| \leq |\mathcal{P}_{2}| + 4k$%
	\COMMENT{VP: Need this small extra step to correct the error.}
 for the second inequality above.
Recall from condition~(vii) that
$X\subseteq V(\mathcal{P}_{2})$ and $X\cap(A^{*}\cup B^{*})=\emptyset$. Thus no paths in $\mathcal{Q}_{1}$
have endpoints in $X$. Moreover, since $t(\mathcal{P}_2)=B$ and $h(\mathcal{P}_2)=A^\prime$, no paths in $\mathcal{Q}_{1}$
have tails in $A^\prime$ or heads in $B$. Together with Lemma~\ref{lem:pathextend-in}(i)--(iii) this implies that
$\mathcal{Q}_{2}$ satisfies (Q2) and
\begin{itemize}
\item[(a$_1$)] $t(\mathcal{Q}_2) \cap A' = h(\mathcal{Q}_2) \cap B = \emptyset$.
\item[(a$_2$)] $h(\mathcal{Q}_2)\cap X=\emptyset$.
\end{itemize}

We will now extend the heads of paths in $\mathcal{Q}_2$ out of $E_B$. 
We do this by applying Lemma~\ref{lem:pathextend-out} to $T'$, $(\mathcal{Q}_{2}\setminus\mathcal{P}_{2})\dot\cup
(\mathcal{Q}_{2}\cap\mathcal{P}_{2})$ with $I:=E_B\setminus X$,
$J:=(E_A\setminus E_B)\cup X\cup A^\prime\cup B$  to form a new path
cover $\mathcal{Q}_{3}$ of $T'$ which will satisfy (Q1)--(Q4). As before, no paths in $\mathcal{P}_{2} \supseteq \mathcal{Q}_{2} \cap \mathcal{P}_{2}$ have endpoints in $I$,
and $F\subseteq E(\mathcal{Q}_{2})$ by (Q1) for $\mathcal{Q}_{2}$.
Moreover, similarly as in~(\ref{eq:help}) we obtain
\begin{align*}
3(|I|+|J|)+2|F| & \le 3|E_B|+3|E_A|+3|X|+3|A^\prime|+3|B|+2|F|\\
 & \le \frac{3}{40}d^{+}+\frac{3}{40}d^{-}+6kt+6t+4ct < \frac{4}{5}d^+.
\end{align*}
(In the final inequality we used our assumption that $d^-\le d^+$.)
Together with (\ref{eq:d+bound-1}) this implies that $d_{T'}^{+}(v) \ge 4d^{+}/5 > 3(|I|+|J|)+2|F|$
for all $v\in I$. Thus the requirements of Lemma~\ref{lem:pathextend-out}
are satisfied, and we can apply the lemma to obtain a path cover $\mathcal{Q}_{3}$
of $T'$. 

By Lemma~\ref{lem:pathextend-out}(iv), $\mathcal{Q}_{3}$ satisfies (Q1). Lemma~\ref{lem:pathextend-out}(v) implies that
\begin{equation}\label{sizeQ3}
|\mathcal{Q}_{3}|\le|\mathcal{Q}_{2}|+|\mathcal{Q}_{2}\setminus \mathcal{P}_{2}|
\stackrel{(\ref{Q12eq})}{\le} |\mathcal{Q}_{2}|+12k\stackrel{(\ref{Q12eq})}{\le} |\mathcal{Q}_{1}|+16k.
\end{equation} 
Similarly, Lemma~\ref{lem:pathextend-out}(vi) implies that
\begin{equation} \label{Q23eq}
|\mathcal{Q}_{3}\cap \mathcal{P}_2|\ge |\mathcal{Q}_{2}\cap \mathcal{P}_2|-|\mathcal{Q}_{2}\setminus \mathcal{P}_2|
\stackrel{(\ref{Q12eq})}{\ge} |\mathcal{Q}_{1}|-20k.
\end{equation}
So $\mathcal{Q}_{3}$ satisfies (Q4).
Lemma~\ref{lem:pathextend-out}(iii) and (Q2) for $\mathcal{Q}_2$ together imply that $\mathcal{Q}_{3}$ satisfies (Q2).
Moreover, (a$_2$) and Lemma~\ref{lem:pathextend-out}(i),(ii) together imply that
no path in $\mathcal{Q}_{3}$ has its head in $(E_B\setminus X)\cup X\supseteq E_B$
and so $\mathcal{Q}_{3}$ satisfies (Q3). Finally, (a$_1$) and
Lemma~\ref{lem:pathextend-out}(ii),(iii) together imply that
\begin{itemize}
\item[(b$_1$)] no paths in $\mathcal{Q}_{3}$ have tails in $A^\prime$ or heads in $B$.
\end{itemize}

We will now extend the paths in $\mathcal{Q}_3 \setminus \mathcal{P}_2$ so that their endpoints lie in $A\cup B'$ rather than $A'\cup B$. More precisely,
if $P\in\mathcal{Q}_{3}\setminus\mathcal{P}_{2}$
has head $a_{i}'\in A^\prime$, then we replace $P$ by $Pa_{i}'a_{i}$
(recall that $a_{i}'a_{i}\in E(T)$ by condition (i) and $a_i\in A\subseteq V(T)\setminus V(\mathcal{Q}_{3})$
by the definition of $N$). If $P\in\mathcal{Q}_{3}\setminus\mathcal{P}_{2}$ has tail $b_{i}\in B$,
we replace $P$ by $b_{i}'b_{i}P$ (recall that $b_{i}'b_{i}\in E(T)$ by condition (iii) and $b_{i}'\in B^\prime\subseteq V(T)\setminus V(\mathcal{Q}_{3})$).
Let $\mathcal{Q}_{4}$ be the path system thus obtained from $\mathcal{Q}_{3}$. 
Let $T'':=T[V(\mathcal{Q}_{4})]$. Then
\[
T'\subseteq T''\subseteq T[V(T')\cup A\cup B^\prime].
\]
and $\mathcal{Q}_{4}$ is a path cover of $T''$ satisfying (Q1)--(Q4) and such that 
\begin{equation}\label{eq:Q4}
|\mathcal{Q}_{4}|=|\mathcal{Q}_{3}| \ \ \ \ \ \ \ \ \text{and} \ \ \ \ \ \ \ \ \mathcal{Q}_{4}\cap \mathcal{P}_2=\mathcal{Q}_{3}\cap \mathcal{P}_2.
\end{equation}
Moreover,
$h(\mathcal{Q}_{4}\setminus\mathcal{P}_{2})\cap A^\prime=\emptyset$ and $t(\mathcal{Q}_{4}\setminus\mathcal{P}_{2})\cap B=\emptyset$.
Together with (b$_1$) this implies that
\begin{itemize}
\item[(c$_1$)] no paths in $\mathcal{Q}_{4}\setminus\mathcal{P}_{2}$ have endpoints in $A^\prime\cup B$.
\end{itemize}
Moreover, by construction of $\mathcal{Q}_{4}$, every vertex $a_i\in V(\mathcal{Q}_{4})\cap A$ is a head of some path
$P\in \mathcal{Q}_{4}\setminus\mathcal{P}_{2}$ and this path $P$ also contains $a'_i$ (so in particular $P_i\notin \mathcal{Q}_{4}\cap \mathcal{P}_{2}$).
Similarly, every vertex in $b'_i\in V(\mathcal{Q}_{4})\cap B^\prime$ is a tail of some path
$P\in \mathcal{Q}_{4}\setminus\mathcal{P}_{2}$ and this path $P$ also contains $b_i$ (in particular $P_i\notin \mathcal{Q}_{4}\cap \mathcal{P}_{2}$).
Thus (Q5) as well as the following assertion hold:
\begin{itemize}
\item[(c$_2$)] no paths in $\mathcal{Q}_{4}$ have heads in $B^\prime$ or tails in $A$.
\end{itemize}

We will now extend the tails of paths in $\mathcal{Q}_4 \setminus \mathcal{P}_2$ out of $A^* \cup B^*$. We do this by applying
the strengthened form of Lemma~\ref{lem:pathextend-in} to $T''$, $(\mathcal{Q}_{4}\setminus\mathcal{P}_{2})\dot\cup (\mathcal{Q}_{4}\cap\mathcal{P}_{2})$
with $I:=B^\prime$, $J:=E_A\cup E_B\cup A^\prime\cup A\cup B$
to form a new path cover $\mathcal{Q}_{5}$ of $T''$ which still satisfies (Q1)--(Q5), and such
that no path in $\mathcal{Q}_{5}\setminus\mathcal{P}_{2}$ has endpoints in
$A^\prime\cup B^\prime\cup B$. Clearly no paths in $\mathcal{P}_{2}\supseteq \mathcal{Q}_{4}\cap\mathcal{P}_{2}$
have tails in $I$, and $F\subseteq E(\mathcal{Q}_{4})$ by (Q1).
By condition (iii) we have $d_{T}^{-}(v)\ge n/3$ for all $v\in I$. Together with
(\ref{eq:d-bound-1}) this implies that $d_{T''}^{-}(v)\ge d_{T'}^{-}(v)\ge n/4$ for
all $v\in I$. Note also that $|V(\mathcal{P}_2)| \le n/20$ by condition~(v). So similarly as in~(\ref{eq:help}), it follows that
\begin{align*}
 3(|I|+|J|) & +2|F|+|V(\mathcal{Q}_{4}\cap\mathcal{P}_{2})|\nonumber\\
 & \le  3(|A^\prime|+|A|+|B^\prime|+|B|+|E_A|+|E_B|)+2|F|+|V(\mathcal{P}_{2})|\nonumber\\
 & \le 12t+\frac{3}{20}d^{+}+4ct+\frac{n}{20} < \frac{n}{4}\le d_{T''}^{-}(v)
\end{align*}
for all $v\in I$. Thus the requirements of the strengthened form of Lemma~\ref{lem:pathextend-in}
are satisfied, and we can apply the lemma to obtain a path cover $\mathcal{Q}_{5}$ of $T''$ such that
$\mathcal{Q}_{5}\cap \mathcal{P}_{2} \supseteq \mathcal{Q}_{4}\cap \mathcal{P}_{2}$. 
Note that Lemma~\ref{lem:pathextend-in}(ii),(iii) imply that the endpoints of $\mathcal{Q}_{5}\setminus (\mathcal{P}_{2} \cap \mathcal{Q}_{4})$
in $J$ are the same as those of $\mathcal{Q}_{4}\setminus \mathcal{P}_{2}$. Together with (c$_1$) this implies that
no paths in $\mathcal{Q}_{5}\setminus (\mathcal{P}_{2} \cap \mathcal{Q}_{4})$
have endpoints in $A^\prime\cup B$.
In particular, this means that $\mathcal{Q}_{5}\cap \mathcal{P}_{2} = \mathcal{Q}_{4}\cap \mathcal{P}_{2}$ and so
\begin{itemize}
\item[(d$_1$)] no paths in $\mathcal{Q}_{5}\setminus \mathcal{P}_{2}$
have endpoints in $A^\prime\cup B$.
\end{itemize}
Thus (Q5) for $\mathcal{Q}_{4}$ implies
that $\mathcal{Q}_{5}$ satisfies (Q5) as well.
Lemma~\ref{lem:pathextend-in}(ii)--(iv), (vi) (strengthened)%
	\COMMENT{VP: I added (vi) (strengthened) because I think we need it to show (Q4) for $\mathcal{Q}_{5}$.}
 and (Q1)--(Q4) for $\mathcal{Q}_4$ together imply that $\mathcal{Q}_{5}$ satisfies (Q1)--(Q4).
Moreover, Lemma~\ref{lem:pathextend-in}(v) implies that
\begin{eqnarray}
|\mathcal{Q}_{5}| & \le & |\mathcal{Q}_{4}|+|\mathcal{Q}_{4}\setminus \mathcal{P}_2|
\stackrel{(\ref{eq:Q4})}{=}|\mathcal{Q}_{3}|+|\mathcal{Q}_{3}\setminus \mathcal{P}_2|
=2|\mathcal{Q}_{3}|-|\mathcal{Q}_{3} \cap \mathcal{P}_2| \nonumber\\
& \stackrel{(\ref{sizeQ3}),(\ref{Q23eq}) }{\le} & |\mathcal{Q}_{1}|+52k.\label{eq:help3}
\end{eqnarray}
By Lemma~\ref{lem:pathextend-in}(i),(ii) and (c$_2$), we can also strengthen (d$_1$) to
\begin{itemize}
\item[(d$_2$)] no paths in $\mathcal{Q}_{5}\setminus \mathcal{P}_{2}$
have endpoints in $A^\prime\cup B^\prime\cup B$ and no paths in $\mathcal{Q}_{5}$ have tails in~$A$.
\end{itemize}

Finally, we will extend the heads of paths in $\mathcal{Q}_5 \setminus \mathcal{P}_2$ out of $A^* \cup B^*$.
We do this by applying the strengthened form of Lemma~\ref{lem:pathextend-out}
to $T''$, $(\mathcal{Q}_{5}\setminus\mathcal{P}_{2})\dot\cup(\mathcal{Q}_{5}\cap\mathcal{P}_{2})$
with $I:=A$, $J:=E_A\cup E_B\cup A^\prime\cup B^\prime\cup B$
to form a new path cover $\mathcal{Q}_{6}$ of $T''$ which
will satisfy (Q1)--(Q7). Clearly no paths in $\mathcal{P}_{2}\supseteq \mathcal{Q}_{5}\cap\mathcal{P}_{2}$
have heads in $I$, and $F\subseteq E(\mathcal{Q}_{5})$ by (Q1).
Similarly as before, condition (i) and (\ref{eq:d+bound-1})
together imply that
\begin{eqnarray*}
 3(|I|+|J|)+2|F|+|V(\mathcal{Q}_{5})\cap\mathcal{P}_{2}|
<  \frac{n}{4}\le d_{T''}^{+}(v)
\end{eqnarray*}
for all $v\in I$. Thus the requirements of the strengthened form of Lemma~\ref{lem:pathextend-out}
are satisfied, and we can apply the lemma to obtain a path cover $\mathcal{Q}_{6}$
of $T''$ such that $\mathcal{Q}_{6}\cap \mathcal{P}_{2}=\mathcal{Q}_{5}\cap \mathcal{P}_{2}$.
(The fact that we have equality follows using a similar argument as in (d$_1$) above.)

Thus (Q5) for $\mathcal{Q}_{5}$ implies that $\mathcal{Q}_{6}$ satisfies (Q5) as well.
Lemma~\ref{lem:pathextend-out}(ii)--(iv), (vi) (strengthened)%
	\COMMENT{VP: Same as above.}
 and (Q1)--(Q4) for $\mathcal{Q}_{5}$ together imply that $\mathcal{Q}_{6}$ satisfies (Q1)--(Q4).
Also, by Lemma~\ref{lem:pathextend-out}(v)  we have
$$
|\mathcal{Q}_{6}| \le |\mathcal{Q}_{5}|+|\mathcal{Q}_{5}\setminus \mathcal{P}_2|
=2|\mathcal{Q}_{5}|-|\mathcal{Q}_{5} \cap \mathcal{P}_2|
\stackrel{({\rm Q4}),(\ref{eq:help3})}{\le } |\mathcal{Q}_{1}|+124k.
$$
So~(Q6) holds. Moreover, by Lemma~\ref{lem:pathextend-out}(i)--(iii), (d$_2$) and the fact that
$\mathcal{Q}_{6}\cap \mathcal{P}_{2}=\mathcal{Q}_{5}\cap \mathcal{P}_{2}$, no paths in
$\mathcal{Q}_{6}\setminus \mathcal{P}_2$ have endpoints in $A^\prime\cup A\cup B^\prime\cup B$.
Since no vertex in $(A^*\cup B^*)\setminus (A^\prime\cup A\cup B^\prime\cup B)$ lies in $V(T'')=V(\mathcal{Q}_{6})$,
this in turn implies (Q7). So the path system $\mathcal{Q}:=\mathcal{Q}_{6}$ is as required in the claim.

\medskip

We will now use the fact that each $A_i$ and each $B_i$ is an almost dominating set in order to extend the paths
in $\mathcal{Q}\setminus\mathcal{P}_{2}$ into those $A_{i}$ and $B_{i}$ which contain the endpoints of paths in $\mathcal{Q}\cap\mathcal{P}_{2}$.
We then use the paths in $\mathcal{Q}\cap\mathcal{P}_{2}$ to join these extended paths
into a long cycle $C$ covering (at least) $N$, and
with $F\subseteq E(C)$. Finally, we will deploy whatever covering
edges we need from $F$ in order to absorb any vertices in $A^{*}\cup B^{*}$
not already covered into $C$. 

Let $\mathcal{R}:=\mathcal{Q}\setminus\mathcal{P}_{2}$ and $\mathcal{S}:=\mathcal{Q}\cap\mathcal{P}_{2}$.
In order to carry out the steps above, we would like to have $|\mathcal{R}| = |\mathcal{S}|$ 
to avoid having any paths in $\mathcal{S}$ left over. So we first split 
the paths in $\mathcal{R}$ until we have exactly $|\mathcal{S}|$
of them. In this process, we wish to preserve (Q1)--(Q3), (Q5) and (Q7).
To show that this can be done, first note that by (Q4) and (Q6), we have
\[
|\mathcal{R}|=|\mathcal{Q}\setminus\mathcal{P}_{2}|\le 144k=t-20k\le |\mathcal{Q}_1|-20k \le|\mathcal{Q}\cap\mathcal{P}_{2}|=|\mathcal{S}|.\label{eq:R=S}
\]
The number of edges in $\mathcal{R}$ which are incident to vertices in $E_A\cup E_B\cup A^{*}\cup B^{*}$,
or which belong to $F$, is bounded above by
\[
2(|E_A|+|E_B|+|A^{*}|+|B^{*}|)+|F|\le\frac{d^{+}}{10}+6ct\le\frac{n}{4}.
\]
On the other hand,
\begin{align*}
|E(\mathcal{R})| & =|V(\mathcal{R})|-|\mathcal{R}|\ge(n-|A^{*}\cup B^{*}|-|V(\mathcal{P}_2)|)-144k\\
& \ge n-2ct-\frac{n}{20}-144k\ge\frac{n}{2}.
\end{align*}
Hence
\[
|E(\mathcal{R})|-2(|E_A|+|E_B|+|A^{*}|+|B^{*}|)-|F|\ge\frac{n}{4}>t\ge|\mathcal{S}|.
\]
We may therefore form a path cover $\mathcal{R}'$ of $T[V(\mathcal{R})]$ with $|\mathcal{R}'|=|\mathcal{S}|$
by greedily removing edges of paths in $\mathcal{R}$ which are neither
incident to $A^{*}\cup B^{*}\cup E_A\cup E_B$ nor elements of~$F$. Then $\mathcal{R}'\cup\mathcal{S}$ satisfies (Q1)--(Q3), (Q5) and (Q7).%
    \COMMENT{$\mathcal{R}'\cap \mathcal{P}_2=\emptyset$ since each path in $\mathcal{R}'$ has at least one endpoint outside $A^*\cup B^*$.}

Next, we extend the paths in $\mathcal{R}'$ into $A^{*}\cup B^{*}$
and join them with the paths in $\mathcal{S}$ to form a long cycle
$C$. By relabeling the $P_i$ if necessary, we may assume that $\mathcal{S}=\{P_1,\dots,P_\ell\}$.
Let $R_1,\dots,R_\ell$ denote the paths in $\mathcal{R}'$ and for each $j\in [\ell]$ let $x_{j}$ be the tail of $R_{j}$
and $y_{j}$  the head of $R_{j}$. Recall from (Q2) and (Q7) that $x_{j}\notin A^{*}\cup B^{*}\cup E_A$.
Hence by condition (ii) there exists $x_{j}'\in A_{j-1}$ with $x_{j}'x_{j}\in E(T)$, where the indices are understood to be modulo~$\ell$.
Similarly $y_{j}\notin A^{*}\cup B^{*}\cup E_B$ by (Q3) and (Q7), so by condition (iv) there exists $y_{j}'\in B_{j}$
with $y_{j}y_{j}'\in E(T)$. Let $R'_j:=x'_{j}x_{j}R_{j}y_{j}y'_{j}$.
If $x_{j}'\ne a_{j-1}'$, then we extend $R'_j$ by adding the edge $a_{j-1}'x'_{j}$.
Similarly, if $y_{j}'\ne b_{j}$ we extend $R'_j$ by adding the edge $y'_{j}b_{j}$.
In all cases, we still denote the resulting path from $a_{j-1}'$ to $b_{j}$ by $R'_j$. 

Recall that $P_j$ is a path from $b_j$ to $a'_j$ for all $j\in [\ell]$.
Moreover, we have $x'_j, y'_j\notin V(\mathcal{Q}\setminus \mathcal{P}_2)=V(\mathcal{R}')$ for all $j\in [\ell]$.
(Indeed, if $x'_j\neq a_j$ this follows since for the oriented graph $T''$ defined in the claim we have $V(T'') \cap A_i  \subseteq \{a_i,a'_i \}$. 
If $x'_j = a_j$, this follows since $P_j \in \mathcal{Q}$ and so (Q5) implies that $a_j \notin V(\mathcal{Q})$.
The argument for $y'_j$ is similar.)
Thus $R_1',\dots,R_\ell'$ are pairwise vertex-disjoint and internally disjoint from the paths in $\mathcal{S}$.
So we can define a cycle $C$ by
\[
C:=R_{1}'P_{1}R_{2}'P_{2}\dots P_{\ell-1}R_{\ell}'P_{\ell}.
\]
Note that $N\subseteq V(C)$ since $\mathcal{R}'\cup\mathcal{S}$
is a path cover of $T''$, and $F\subseteq E(C)$ by (Q1).
Recall from condition (vi) that $F$ consists of covering edges $e_v$
for all $v\in A^{*}\cup B^{*}$ and that these $e_v$ are pairwise distinct.
Thus each $e_v$ lies on $C$ and so neither of the two activating edges of $e_v$ can lie on~$C$. Writing
$e_v = x_vy_v$, it follows from
these observations that we may form a new cycle $C'$ by replacing $x_{v}y_{v}$
by $x_{v}vy_{v}$ in $C$ for all $v\in(A^{*}\cup B^{*})\setminus V(C)$.
Then $C'$ is a Hamilton cycle of~$T$, as desired.
\end{proof}


\section{Finding many edge-disjoint Hamilton cycles in a good tournament\label{sec:manycycles}}

In the proof of Theorem~\ref{thm:mainresult}, we will find the edge-disjoint Hamilton cycles in a given highly-linked
tournament by repeatedly applying Lemma~\ref{lem:mainengine}. In each application, we will need to set up all the dominating sets and paths required
by Lemma~\ref{lem:mainengine}. The following definition encapsulates this idea. (Recall that $\Int(P)$ denotes the interior of a path~$P$.)
\begin{defn}
\label{def:good}
We say that a tournament $T$ is \emph{$(C,k,t,c)$-good}
if it contains vertex sets $A_{1}^{1},\dots,A_{k}^{t}$, $B_{1}^{1},\dots,B_{k}^{t}$, 
$E_{A,1},\dots,E_{A,k}$, $E_{B,1},\dots,E_{B,k}$,
edge sets $F_{1},\dots,F_{k}$, and paths $P_{1}^{1},\dots,P_{k}^{t}$
such that the following statements hold, where $A_{i}^{*}:=A_{i}^{1}\cup\dots\cup A_{i}^{t}$,
$A^{*}:=A_{1}^{*}\cup\dots\cup A_{k}^{*}$, $B_{i}^{*}:=B_{i}^{1}\cup\dots\cup B_{i}^{t}$,
and $B^{*}:=B_{1}^{*}\cup\dots\cup B_{k}^{*}$:
\end{defn}
\begin{enumerate}
\item [(G1)]The sets $A_{1}^{1},\dots,A_{k}^{t}$ are disjoint and $2\le |A_{i}^{\ell}|\le c$ for all $i\in [k]$ and $\ell\in [t]$.
Moreover, each $T[A_{i}^{\ell}]$ is a transitive tournament whose head has out-degree at least $2n/5$ in $T$.
Write $A:=\{h(T[A_{i}^{\ell}]):i\in[k],\ell\in[t]\}$.
\item [(G2)]The sets $B_{1}^{1},\dots,B_{k}^{t}$ are disjoint from each other and 
from $A^*$, and $2\le |B_{i}^{\ell}|\le c$ for all $i\in [k]$ and $\ell\in [t]$. Moreover, each $T[B_{i}^{\ell}]$ is a transitive
tournament whose tail has in-degree at least $2n/5$ in $T$. Write $B':=\{t(T[B_{i}^{\ell}]):i\in[k],\ell\in[t]\}$.
\item [(G3)]Write $d_{-}:=\min\{d^{-}(v):v\in V(T)\setminus(A\cup B')\}$.
Each $A_{i}^{\ell}$ out-dominates $V(T)\setminus (A^{*}\cup B^{*}\cup E_{A,i})$.
Moreover, $|E_{A,i}|\le d_{-}/50$ and $E_{A,i}\cap(A_{i}^{*}\cup B_{i}^{*})=\emptyset$
for all $i\in [k]$.%
   \COMMENT{This is not a typo -- we don't require $A_i^\ell$ to out-dominate any of $A^* \cup B^*$, and in general
we can't require it while preserving disjointness.}
\item [(G4)]Write $d_{+}:=\min\{d^{+}(v):v\in V(T)\setminus(A\cup B')\}$.
Each $B_{i}^{\ell}$ in-dominates $V(T)\setminus (A^{*}\cup B^{*}\cup E_{B,i})$.
Moreover, $|E_{B,i}|\le d_{+}/50$ and $E_{B,i}\cap(A_{i}^{*}\cup B_{i}^{*})=\emptyset$
for all $i\in [k]$.
\item [(G5)]Each $P_{i}^{\ell}$ is a path from the head of $T[B_{i}^{\ell}]$ to
the tail of $T[A_{i}^{\ell}]$. For each $i\in [k]$, the paths $P_{i}^{1},\dots,P_{i}^{t}$ are vertex-disjoint
and $|P_{1}^{1}\cup\dots\cup P_{k}^{t}|\le n/20$.
For all $i\neq j$ and all $\ell, m \in [t]$, $P_{i}^{\ell}$ and $P_{j}^{m}$ are edge-disjoint%
\COMMENT{TO DO: Previously also had $V(P_{i}^{\ell})\cap V(P_{j}^{m})\subseteq A\cup B'$. Make sure this really isn't needed. (This should be fine - JL)
Also, we previously only had
$V(\Int(P_{i}^{\ell}))\cap (A^{*}\cup B^{*})\subseteq (A\cup B')$. But I don't see
how we then get that $P_{i}^{1},\dots,P_{i}^{t}$ are internally disjoint from $A_i^*\cup B_i^*$.}
and
$$
V(\Int(P_{i}^{\ell}))\cap (A^{*}\cup B^{*})\subseteq (A\cup B')\setminus (A_i^*\cup B_i^*).
$$   
\item [(G6)]$F_{i}\subseteq E(P_{i}^{t})$ and $(A\cup B')\setminus(A_{i}^{*}\cup B_{i}^{*})\subseteq V(P_{i}^{t})$
for all $i\in [k]$.%
\COMMENT{The last point is important because these vertices are not generally dominated by $A_i^*$ and $B_i^*$,
and they have unusually low in- or out-degrees. We need these to be covered by paths before we start Lemma \ref{lem:mainengine},
or they'd potentially stop us extending out of the exceptional sets.}
\item [(G7)]The set $F_{1}\cup\dots\cup F_{k}$ is a matching in $T-(A^{*}\cup B^{*})$. For all $i\in [k]$
we have $F_i = \{e_v:v \in A^*_i \cup B^*_i\}$, where $e_v$ is a covering edge for $v$ and
$e_v \ne e_{v'}$ whenever $v \ne v'$.
Moreover, for each $i\in [k]$, let $F^{\rm act}_{i}$ be the set of activating edges corresponding to the
covering edges in $F_i$. Then $F^{\rm act}_{i}\cap E(P_{j}^{\ell})=\emptyset$ for all $i,j \in [k]$ and all $\ell\in [t]$.%
   \COMMENT{Previously also had that $F^{\rm act}_{i}\cap F^{\rm act}_{j}=\emptyset$ for all distinct $i,j\in [k]$.
But this follows from the other conditions, so I deleted it.}
\item [(G8)]We have $\delta^{0}(T)\ge Ck^{2}\log k$.
\end{enumerate}
For convenience, we collect the various disjointness conditions of
Definition~\ref{def:good} into a single statement.
\begin{enumerate}
\item [(G9)]
\begin{itemize}
\item The sets $A_{1}^{1},\dots,A_{k}^{t}$, $B_{1}^{1},\dots,B_{k}^{t}$
are disjoint. 
\item $(E_{A,i}\cup E_{B,i})\cap(A_{i}^{*}\cup B_{i}^{*})=\emptyset$
for all $i\in [k]$. 
\item $F_{1}\cup\dots\cup F_{k}$ is a matching in $T-(A^{*}\cup B^{*})$.
\item For each $i\in [k]$, the paths $P_{i}^{1},\dots,P_{i}^{t}$ are vertex-disjoint.
\item For all $i\neq j$ and all $\ell, m\in [t]$, $P_{i}^{\ell}$ and $P_{j}^{m}$ are edge-disjoint
and $V(\Int(P_{i}^{\ell}))\cap (A^{*}\cup B^{*})\subseteq (A\cup B')\setminus (A_i^*\cup B_i^*)$. In particular, $P_{i}^{1},\dots,P_{i}^{t}$
are internally disjoint from $A_i^*\cup B_i^*$.%
\COMMENT{I think adding this is worth it -- otherwise we end up with quite long lists in the proof of Lemma \ref{lem:good-embed}.}
\end{itemize}
\end{enumerate}

The next lemma shows that for suitable parameters $C$, $t=t(k)$ and $c=c(k)$, every $(C,k,t,c)$-good tournament contains $k$ edge-disjoint
Hamilton cycles.%
   \COMMENT{The proof works if $t(k)=\Theta(k)$ and $c(k)=\Theta(\log k)$, as long as $Ck^{2}\log k\gg ktc$.}
In the next section we then show that there exists a constant $C'>0$ such that
any $C'k^{2}\log k$-linked tournament is $(C,k,t,c)$-good (see Lemma~\ref{lem:good}).
These two results together immediately imply Theorem~\ref{thm:mainresult}.

As mentioned at the beginning of this section, in order to prove Lemma~\ref{lem:good-embed} we will
apply Lemma~\ref{lem:mainengine} $k$ times. In the notation for Definition~\ref{def:good}, our convention is
that the sets with subscript $i$ will be used in the $i$th application of Lemma~\ref{lem:mainengine}
to find the $i$th Hamilton cycle. 

\begin{lem}
\label{lem:good-embed}
Let $C:=10^{7}$, $k\ge20$, $t:=164k$, $c:=\lceil \log 50t+1\rceil$.
Then any $(C,k,t,c)$-good tournament contains $k$ edge-disjoint
Hamilton cycles.
\end{lem}
\begin{proof}
Let $T$ be a $(C,k,t,c)$-good tournament, and let $n:=|T|$. Let
$A_{1}^{1},\dots,A_{k}^{t}$, 
$B_{1}^{1},\dots,B_{k}^{t}$,
$E_{A,1},\dots,E_{A,k}$, $E_{B,1},\dots,E_{B,k}$,
$F_{1},\dots,F_{k}$, $P_{1}^{1},\dots,P_{k}^{t}$,
$d_{-}$ and $d_{+}$ be as in Definition~\ref{def:good}. 
(Note that this also implicitly defines sets $A_1^*, \dots, A_k^*$, $A^*$, $A$, 
$B_1^*, \dots, B_k^*$, $B^*$, $B'$, and $F_1^{\textnormal{act}}, \dots, F_k^{\textnormal{act}}$
as in Definition~\ref{def:good}.)
Our aim is to apply Lemma~\ref{lem:mainengine} repeatedly to find $k$ edge-disjoint
Hamilton cycles. So suppose that for some $i\in [k]$ 
we have already found edge-disjoint Hamilton cycles $C_{1},\dots,C_{i-1}$
such that the following conditions hold:
\begin{itemize}
\item[(a)] $C_{1},\dots,C_{i-1}$ are edge-disjoint from $T[A_{j}^{\ell}],T[B_{j}^{\ell}]$ and $P_{j}^{\ell}$
for all $i\le j\le k$ and all $\ell\in[t]$.
\item[(b)] $E(C_{1}\cup\dots\cup C_{i-1})\cap F^{\rm act}_{j}=\emptyset$ for all $i\le j\le k$.
\end{itemize}
Intuitively, these conditions guarantee that none of the edges we will need in order to find $C_i, \dots, C_k$ are contained in $C_{1},\ldots,C_{i-1}$. 
We have to show that $T-C_{1}-\dots-C_{i-1}$ contains a Hamilton cycle $C_{i}$
which satisfies (a) and (b) (with $i$ replaced by $i+1$). 

Define
\begin{align*}
T_{i} & :=T-\left(\bigcup_{j<i}C_{j}\cup\bigcup_{j>i}F^{\rm act}_{j}\right)-\bigcup_{j>i,\ \ell\in[t]}(P_{j}^{\ell}\cup T[A_{j}^{\ell}]
\cup T[B_{j}^{\ell}]),\\
E_{A,i}^{\prime} & :=E_{A,i}\cup\left(\left(\bigcup_{j<i}N_{C_{j}}^{+}(A_{i}^{*})\cup\bigcup_{j>i,\ \ell\in[t]}N_{P_{j}^{\ell}}^{+}(A_{i}^{*})\cup A^{*}\cup B^{*}\right)\setminus(A_{i}^{*}\cup B_{i}^{*})\right),\\
E_{B,i}^{\prime} & :=E_{B,i}\cup\left(\left(\bigcup_{j<i}N_{C_{j}}^{-}(B_{i}^{*})\cup\bigcup_{j>i,\ \ell\in[t]}N_{P_{j}^{\ell}}^{-}(B_{i}^{*})\cup A^{*}\cup B^{*}\right)\setminus(A_{i}^{*}\cup B_{i}^{*})\right),\\
X_{i} & :=(A\cup B')\setminus(A_{i}^{*}\cup B_{i}^{*}).
\end{align*}
Then it suffices to find a Hamilton cycle $C_{i}$ of $T_{i}$.
We will do so by applying Lemma~\ref{lem:mainengine} to $T_{i}$,
$A_{i}^1,\dots,A_i^t$, $B_{i}^1,\dots,B_i^t$, $P_i^1,\dots, P_i^t$, $E_{A,i}^{\prime}$, $E_{B,i}^{\prime}$, $F_i$ 
and $X_{i}$. It therefore suffices to verify that the conditions of
Lemma~\ref{lem:mainengine} hold. 

We claim that for each $v\in V(T_{i})$, we have
\begin{equation}
d_{T_{i}}^{+}(v)\ge d_{T}^{+}(v)-(i-1)-(k-i)-1-c>d_{T}^{+}(v)-2k.\label{eq:d+bound-2}
\COMMENT{VP: Use $k \geq 20$ here}
\end{equation}
Indeed, it is immediate that $d_{C_{1}\cup\dots\cup C_{i-1}}^{+}(v)=i-1$.
Since by (G9) for each $j>i$ the paths $P_{j}^{1},\dots,P_{j}^{t}$ are vertex-disjoint,
$v$ is covered by at most $k-i$ of the paths $P_{i+1}^{1},\dots,P_{k}^{t}$
and hence $d_{P_{i+1}^{1}\cup\dots\cup P_{k}^{t}}^{+}(v)\le k-i$.
Recall from (G7) that $F_1\cup \dots\cup F_k$ consists of one covering edge $e_v$ for each $v\in A^*\cup B^*$.
Moreover, by (G9) the set $F_{1}\cup\dots\cup F_{k}$ is a matching in $T-(A^*\cup B^*)$ and $A_{1}^{1},\dots,A_{k}^{t},B_{1}^{1},\dots,B_{k}^{t}$
are all disjoint. Thus the digraph with edge set $F^{\rm act}_{1}\cup\dots\cup F^{\rm act}_{k}$
is a disjoint union of directed paths of length~two and therefore has maximum
out-degree one. Finally, since $A_{1}^{1},\dots,A_{k}^{t},B_{1}^{1},\dots,B_{k}^{t}$
are disjoint, $v$ belongs to at most one of $T[A_{1}^{1}],\dots,T[A_{k}^{t}],T[B_{1}^{1}],\dots,T[B_{k}^{t}]$.
Moreover, $\Delta^{+}(T[A_{j}^{\ell}]),\Delta^{+}(T[B_{j}^{\ell}])\le c$
for all $j>i$ and all $\ell\in [t]$ by (G1) and (G2). So (\ref{eq:d+bound-2}) follows.
Similarly, we have
\begin{equation}
d_{T_{i}}^{-}(v)>d_{T}^{-}(v)-2k.\label{eq:d-bound-2}
\end{equation}
In particular, $\delta(T_{i})>n-4k$, as required by Lemma~\ref{lem:mainengine}.

We have $\delta^{0}(T)>Ck^{2}$ by (G8), and hence $\delta^{0}(T_{i})>10^{6}k^{2}$
as required by Lemma~\ref{lem:mainengine}. The disjointness conditions
of Lemma~\ref{lem:mainengine} are satisfied by (G9) and the definition of $X_i$. Since $V(T_i) = V(T)$, it is immediate
that $A_i^1, \dots, A_i^t,B_i^1, \dots, B_i^t,X_i \subseteq V(T_i)$. We claim that $P_i^1,\dots,P_i^t\subseteq T_i$. Indeed,
by~(a) and (G5), each $P_i^\ell$ is edge-disjoint from $C_1\cup\dots\cup C_{i-1}$ and from $P_j^{m}$
for all $j>i$ and all $m\in [t]$. By (G7), each $P_i^\ell$ is edge-disjoint from $F_1^{\rm act}\cup \dots\cup F_k^{\rm act}$.
Moreover, by (G5), each $P_i^\ell$ is edge-disjoint from $T[A_j^{m}]\cup T[B_j^m]$ for all $j>i$ and%
    \COMMENT{This holds since (G5) implies that eg $|V(P_i^\ell) \cap A_j^{m}| \le 1$ for all $j>i$, $m\in [t]$}
all $m\in [t]$. Altogether this implies that $P_i^1,\dots,P_i^t\subseteq T_i$. 
We have $F_i \subseteq E(P_i^t) \subseteq E(T_i)$ by (G6). It therefore suffices to prove that
conditions~(i)--(vii) of Lemma~\ref{lem:mainengine} hold.

Condition~(v) follows from (G5). Condition~(vi) follows from (G6) and (G7). 
(Note that (G7) implies that 
$F^{\rm act}_{i}\cap F^{\rm act}_{j}=\emptyset$ for all $i \ne j$. So (G7), (b) and the definition of $T_i$ imply that
$F^{\rm act}_{i} \subseteq T_i$.)
By (G6) we have $X_{i}\subseteq V(P_{i}^{t})$ and by (G1) and (G2) we
have $|X_{i}|\le|A\cup B'|=2kt$, so condition (vii) holds too.

It therefore remains to verify conditions (i)--(iv).
We first check~(i). We have $2\le |A_{i}^{\ell}|\le c$ by (G1).
Moreover, we claim that  $T_{i}[A_{i}^{\ell}]=T[A_{i}^{\ell}]$ for all $\ell\in [t]$.
Indeed, to see this, note that $C_{1},\dots,C_{i-1}$ are edge-disjoint from $T[A_{i}^{\ell}]$
by (a); by (G9) for all $j>i$ and all $m\in [t]$ each path $P_{j}^{m}$ and each
$T[A_{j}^{m}]$, $T[B_{j}^{m}]$ is edge-disjoint from $T[A_{i}^{\ell}]$; by (G7) all edges in $F^{\rm act}_{j}$
for $j>i$ are incident to a vertex in $A_{j}^{*}\cup B_{j}^{*}$,
and hence  by (G9) none of these edges belongs to $T[A_{i}^{\ell}]$. Thus $T_{i}[A_{i}^{\ell}]=T[A_{i}^{\ell}]$
is a transitive tournament by (G1). Finally,
by (G1) the head of each $T[A_{i}^{\ell}]$ has out-degree at least $2n/5$
in $T$, and so  by (\ref{eq:d+bound-2}) out-degree at least $n/3$ in $T_{i}$. 
Hence condition (i) of Lemma~\ref{lem:mainengine}
is satisfied. A similar argument shows that condition (iii) of Lemma~\ref{lem:mainengine} is also satisfied.%
    \COMMENT{Indeed, similarly as before for all $j>i$ we have $T[B_{i}^{j}]=T_{i}[B_{i}^{j}]$.
Moreover, $2\le |B_{i}^{j}|\le c$ by (G2), and the tail of $B_{i}^{j}$ has in-degree
at least $2n/5$ in $T$ by (G2) and hence in-degree at least $n/3$
in $T_{i}$ by (\ref{eq:d-bound-2}). Hence condition (iii) of
Lemma~\ref{lem:mainengine} is also satisfied.}

We will next verify that condition (ii) of Lemma~\ref{lem:mainengine} holds too.
(G9) and the definition of $E_{A,i}^{\prime}$ together imply
that $E_{A,i}^{\prime}\cap (A_{i}^{*}\cup B_{i}^{*})=\emptyset$.
By (G3), each $A_{i}^{\ell}$ out-dominates $V(T)\setminus (A^{*}\cup B^{*}\cup E_{A,i})$
in $T$, and hence out-dominates
$V(T_{i})\setminus (A^{*}\cup B^{*}\cup E_{A,i}\cup N_{T-T_{i}}^{+}(A_{i}^{*}))$ in~$T_i$.
However, it follows from (G9) that for all $j>i$ and all $\ell,m\in [t]$, no edge in $F^{\rm act}_j$ has an endpoint in $A_i^{\ell}$ and
that $A_i^{\ell} \cap A_j^m = A_i^{\ell} \cap B_j^m = \emptyset$. Hence by (G9) we have that
\[
N_{T-T_{i}}^{+}(A_{i}^{*})=\bigcup_{j<i}N_{C_{j}}^{+}(A_{i}^{*})\cup\bigcup_{j>i,\ \ell\in[t]}N_{P_{j}^{\ell}}^{+}(A_{i}^{*}).
\]
 It therefore follows from the definitions
of $E_{A,i}^{\prime}$ and $T_{i}$ that $A_{i}^{\ell}$ out-dominates
$V(T_{i})\setminus (A_{i}^{*}\cup B_{i}^{*}\cup E_{A,i}^{\prime})$ in $T_i$ for all $\ell \in [t]$.

So in order to check that condition (ii) of Lemma~\ref{lem:mainengine} holds, it remains only to bound $|E_{A,i}^{\prime}|$ from above.
To do this, first note that by (G9), each vertex in $A_{i}^{*}$ is contained in at most $k-i$ of the
paths $P_{i+1}^{1},\dots,P_{k}^{t}$. Moreover, $|E_{A,i}|\le d_{-}/50$ by~(G3).
It therefore follows from the definition of $E_{A,i}^{\prime}$, (G1)
and (G2) that%
     \COMMENT{The final inequality below is a bottleneck for $C$}
\begin{align*}
|E_{A,i}^{\prime}| & \le|E_{A,i}|+\left|\bigcup_{j<i}N_{C_{i}}^{+}(A_{i}^{*})\right|+
\left|\bigcup_{j>i,\ \ell\in[t]}N_{P_{j}^{\ell}}^{+}(A_{i}^{*})\right|+|A^{*}|+|B^{*}|\\
 & \le\frac{d_{-}}{50}+(i-1)|A_{i}^{*}|+(k-i)|A_{i}^{*}|+2kct
 \le\frac{d_{-}}{50}+kct+2kct\le\frac{d_{-}}{45}. 
 \end{align*}
The last inequality follows since $d_{-}\ge\delta^{0}(T)\ge Ck^{2}\log k$
by (G8). Since $E_{A,i}^{\prime}$ is disjoint from $A_{i}^{*}\cup B_{i}^{*}$,
we have $E_{A,i}^{\prime}\setminus X_{i}=E_{A,i}^{\prime}\setminus(A\cup B')$.
Hence for all $v\in E_{A,i}^{\prime}\setminus X_{i}$ we have
\[
d_{T_{i}}^{-}(v)\stackrel{(\ref{eq:d-bound-2})}{\ge} d_{T}^{-}(v)-2k \stackrel{{\rm (G3)}}{\ge} d_{-}-2k \ge \frac{19}{20} d_-
\]
and so
\[
|E_{A,i}^{\prime}|\le\frac{d_{-}}{45}\le\frac{1}{40}\min\{d_{T_{i}}^{-}(v):v\in E_{A,i}^{\prime}\setminus X_{i}\}.
\]
This shows that condition (ii) of Lemma~\ref{lem:mainengine} is satisfied.
The argument that (iv) holds is similar.%
   \COMMENT{It is immediate from (G9) and the definition of $E_{B,i}^{\prime}$
that $E_{B,i}^{\prime}\cap (A_{i}^{*}\cup B_{i}^{*})=\emptyset$.
By (G4), each $B_{i}^{\ell}$ in-dominates $V(T)\setminus (A^{*}\cup B^{*}\cup E_{B,i})$
in $T$, and hence each $B_{i}^{\ell}$ in-dominates $V(T_{i})\setminus (A^{*}\cup B^{*}\cup E_{B,i}\cup N_{T-T_{i}}^{-}(B_{i}^{*}))$
in $T_i$. It therefore follows as above that each $B_{i}^{\ell}$ in-dominates
$V(T')\setminus (A_{i}^{*}\cup B_{i}^{*}\cup E_{B,i}^{\prime})$ in $T_i$. It remains only
to bound $|E_{B,i}^{\prime}|$ above. From (G5), each vertex in $B_{i}^{*}$ is contained in at most $k-i$
paths $P_{i+1}^{1},\dots,P_{k}^{t}$. From (G4), we have $|E_{B,i}|\le d_{+}/50$.
It therefore follows from the definition of $E_{B,i}^{\prime}$, (G1)
and (G2) that
\begin{eqnarray*}
|E_{B,i}^{\prime}| & \le & |E_{B,i}|+\left|\bigcup_{j<i}N_{C_{i}}^{-}(B_{i}^{*})\right|+
\left|\bigcup_{j>i,\ \ell\in[t]}N_{P_{j}^{\ell}}^{-}(B_{i}^{*})\right|+|A^{*}|+|B^{*}|\\
 & \le & \frac{d_{+}}{50}+(i-1)|B_{i}^{*}|+(k-i)|B_{i}^{*}|+2kct\le\frac{d_{+}}{45}.
\end{eqnarray*}
Since $E_{B,i}^{\prime}$ is disjoint from $A_{i}^{*}\cup B_{i}^{*}$,
we have $E_{B,i}^{\prime}\setminus X_{i}=E_{B,i}^{\prime}\setminus(A\cup B)$.
Hence by (G4) and (\ref{eq:d+bound-2}), for all $v\in E_{B,i}^{\prime}\setminus X_{i}$
we have
\[
d_{T'}^{+}(v)\ge d_{T}^{+}(v)-2k\ge d_{+}-2k.
\]
Hence
\[
|E_{B,i}^{\prime}|\le\frac{d_{+}}{45}\le\frac{1}{40}\min\{d_{T'}^{+}(v):v\in E_{B,i}^{\prime}\setminus X_{i}\},
\]
and condition (iv) of Lemma~\ref{lem:mainengine} is satisfied.}
We may therefore apply Lemma~\ref{lem:mainengine} to find a Hamilton cycle $C_i$ in $T_i$ as desired.
\end{proof}


\section{Highly-linked tournaments are good\label{sec:linkagegood}}

The aim of this section is to prove that any sufficiently highly-linked tournament is $(C,k,t,c)$-good. We first
show that it is very easy to find covering edges for any given vertex -- we will use the following
lemma to find matchings $F_{1},\dots,F_{k}$ consisting of covering edges
as in Definition~\ref{def:good}.

\begin{lem}
\label{lem:coveringedge}Suppose that $T$ is a strongly $2$-connected tournament,
and $v\in V(T)$. Then there exists a covering edge for $v$.\end{lem}
\begin{proof}
Since $T$ is strongly connected and $|T|>1$, we have $N^{+}(v),N^{-}(v)\ne\emptyset$.
Since $T-v$ is strongly connected, there is an edge $xy$ from $N^{-}(v)$
to $N^{+}(v)$. But then $xv,vy\in E(T)$, so $xy$ is a covering
edge for $v$, as desired.
\end{proof}

The next lemma will be used to obtain paths $P_{1}^{1},\dots,P_{k}^{t}$
as in Definition~\ref{def:good}. Recall that we require $F_{i}\subseteq E(P_{i}^{t})$ and
$(A\cup B')\setminus(A_{i}^{*}\cup B_{i}^{*})\subseteq V(P_{i}^{t})$
for all $i\in [k]$. We will ensure the latter requirement by first covering $(A\cup B')\setminus(A_{i}^{*}\cup B_{i}^{*})$
with few paths and then linking these paths together -- hence the form of the lemma.

\begin{lem}
\label{lem:shortlinkagewithpaths}Let $s\in\mathbb{N}$, and let
$T$ be a digraph. Let $x_{1},\dots, x_k$, $y_{1},\dots,y_{k}$ be distinct vertices of $T$, and let $\mathcal{Q}_{1},\dots,\mathcal{Q}_{k}$
be (possibly empty) path systems in $T-\{x_{1},\dots, x_k,y_{1},\dots,y_{k}\}$ with $E(\mathcal{Q}_{i})\cap E(\mathcal{Q}_{j})=\emptyset$
whenever $i\neq j$. Write
\begin{equation}\label{eq:defm}
m:=k+\sum_{i=1}^{k}|\mathcal{Q}_{i}|+\left|\bigcup_{i=1}^{k}V(\mathcal{Q}_{i})\right|,
\end{equation}
and suppose that $T$ is $2s m$-linked. Then there exist edge-disjoint paths $P_{1},\dots,P_{k}\subseteq T$
satisfying the following properties:
\begin{enumerate}
\item $P_{i}$ is a path from $x_{i}$ to $y_{i}$ for all $i\in [k]$.
\item $Q\subseteq P_{i}$ for  all $Q\in\mathcal{Q}_{i}$ and all $i\in [k]$.
\item $V(P_{i})\cap V(P_{j})\subseteq V(\mathcal{Q}_{i})\cap V(\mathcal{Q}_{j})$ for all $i\neq j$.
\item $|P_{1}\cup\dots\cup P_{k}|\le|T|/s+|V(\mathcal{Q}_{1})\cup\dots\cup V(\mathcal{Q}_{k})|$.%
         \COMMENT{Can't write $V(\mathcal{Q}_{1}\cup\dots\cup \mathcal{Q}_{k})$
since $\mathcal{Q}_1 \cup \dots \mathcal{Q}_k$ need not be a path system.}
\end{enumerate}
\end{lem}
\begin{proof}
For all $i\in[k]$, let $a_{i}^{1} \dots b_{i}^{1},\dots,a_{i}^{t_{i}} \dots b_{i}^{t_{i}}$ denote the paths in $\mathcal{Q}_{i}$.
Let $F\subseteq E(T)$ denote the set of all those edges which form a path of length one in $\mathcal{Q}_{1}\cup \dots\cup \mathcal{Q}_{k}$.
Let
\[
T':=T\biggl[\biggl(V(T)\setminus\bigcup_{i=1}^{k}V(\mathcal{Q}_{i})\biggr)\cup\bigcup_{i=1}^{k}\bigcup_{j=1}^{t_{i}}\{a_{i}^{j},b_{i}^{j}\}\biggr]-F.
\]
Note that $E(T')\cap (E(\mathcal{Q}_1)\cup\dots\cup E(\mathcal{Q}_k))=\emptyset$.
Define sets $X_{1},\dots,X_{k}$ of ordered pairs of vertices of $T'$
by
\[
X_{i}:=\begin{cases}
\{(x_{i},a_{i}^{1}),(b_{i}^{1},a_{i}^{2}),\dots,(b_{i}^{t_{i}-1},a_{i}^{t_{i}}),(b_{i}^{t_{i}},y_{i})\}, & \text{if }\mathcal{Q}_{i}\ne\emptyset,\\
\{(x_{i},y_{i})\} & \text{if } \mathcal{Q}_{i}=\emptyset,
\end{cases}
\]
and let $X:=X_{1}\cup\dots\cup X_{k}$. 
Let $\ell:= 2sm-2s|X|$. Since $|V(T)\setminus V(T')|+|F|\le|V(\mathcal{Q}_{1})\cup\dots\cup V(\mathcal{Q}_{k})|$
and $|X|=k+\sum_{i=1}^{k}|\mathcal{Q}_{i}|$, it follows that
\[2\ell = 4s(m - |X|) \stackrel{(\ref{eq:defm})}{=} 4s\left|\bigcup_{i=1}^k V(\mathcal{Q}_i)\right| \ge |V(T)\setminus V(T')| + 2|F|.\]
Thus by Proposition~\ref{prop:robustlinkage}, $T'$ is $2s|X|$-linked.
We may therefore apply Lemma~\ref{lem:shortlinkage} to $X$ in order to obtain, for each $i\in [k]$, a path system $\mathcal{P}_i$
whose paths link the pairs in $X_i$ and such that whenever $i\neq j$, we have $E(\mathcal{P}_{i})\cap E(\mathcal{P}_{j})=\emptyset$ and 
$V(\mathcal{P}_{i})\cap V(\mathcal{P}_{j})$ consists of exactly the vertices that lie in a pair in both $X_i$ and $X_j$. Let $P_i$
be the path obtained from the union of all paths in $\mathcal{P}_i$ and
all paths in $\mathcal{Q}_i$. Then $P_1,\dots,P_k$ are edge-disjoint paths satisfying (i)--(iv).
\end{proof}

The next lemma shows that given a vertex $v$ in a tournament~$T$, we can find a small transitive subtournament whose head is~$v$
and which out-dominates almost all vertices of $T$.

\begin{lem}
\label{lem:out-dom-trans}Let $T$ be a tournament on $n$ vertices,
let $v\in V(T)$, and suppose that $c\in\mathbb{N}$ satisfies $2\le c\le\log d^{-}(v)-1$.
Then there exist disjoint sets $A,E\subseteq V(T)$ such that the
following properties hold:
\begin{enumerate}
\item $2\le |A|\le c$ and $T[A]$ is a transitive tournament with head $v$.
\item $A$ out-dominates $V(T)\setminus (A\cup E)$.
\item $|E|\le(1/2)^{c-1}d^{-}(v)$. 
\end{enumerate}
\end{lem}
The fact that the bound in (iii) depends on $d^-(v)$ is crucial:
for instance, we can apply Lemma~\ref{lem:out-dom-trans} with $v$ being the vertex of lowest 
in-degree. Then (iii) implies that the `exceptional set' $|E|$ is much smaller than 
$d^-(v) \le d^-(w)$ for any $w \in E$. So while $w$ is not dominated by $A$ directly, 
it is dominated by many vertices outside $E$. This will make it possible to cover 
$E$ by paths whose endpoints lie outside $E$. (More formally, the lemma is used to ensure (G3), which in turn is 
used for (Q2) in the proof of Lemma~\ref{lem:mainengine}).
\begin{proof}
Let $v_{1}:=v$. We will find $A$ by repeatedly choosing vertices $v_1,\dots,v_i$ such that
the size of their common in-neighbourhood is minimised at each
step. More precisely, let $A_{1}:=\{v_{1}\}$. Suppose that for some $i<c$ we have already found a set
$A_{i}=\{v_{1},\dots,v_{i}\}$ such that $T[A_i]$ is a transitive tournament with head $v_{1}$, and such that the
common in-neighbourhood $E_{i}$ of $v_{1},\dots,v_{i}$ satisfies
\[
|E_{i}|\le\frac{1}{2^{i-1}}d^{-}(v).
\]
Note that these conditions are satisfied for $i=1$. Moreover, 
note that $E_{i}$ is the set of all those vertices in $T-A_i$ which are not out-dominated by $A_{i}$.
If $|E_{i}|<4$, then we have
\begin{equation}\label{eq:sizeEi}
|E_{i}|<4=\frac{1}{2^{\log d^{-}(v)-2}}d^{-}(v)\le\frac{1}{2^{c-1}}d^{-}(v),
\end{equation}
and so $A_{i}$ satisfies~(i)--(iii). (Note that $|A_i|\ge 2$ since the assumptions imply that $d^-(v) \ge 8$.) 
Thus in this case we can take $A:=A_i$ and $E:=E_i$.

So suppose next that $|E_{i}|\ge4$. In this case we will extend $A_i$ to $A_{i+1}$ by adding a suitable vertex $v_{i+1}$.
By Proposition~\ref{prop:largedegreevertex}, $E_{i}$ contains a vertex $v_{i+1}$ of in-degree at most $|E_{i}|/2$ in
$T[E_{i}]$. Let $A_{i+1}:=\{v_{1},\dots,v_{i+1}\}$ and let $E_{i+1}$ be the common in-neighbourhood of $v_{1},\dots,v_{i+1}$.
Then $T[A_{i+1}]$ is a transitive tournament with head $v_1$ and 
\[
|E_{i+1}|\le\frac{1}{2}|E_{i}|\le\frac{1}{2^{i}}d^{-}(v).
\]
By repeating this construction, either we will find $|E_{i}|<4$ for
some $i< c$ (and therefore take $A:=A_{i}$ and $E:=E_{i}$) or we
will obtain sets $A_c$ and $E_c$ satisfying
(i)--(iii).
\end{proof}

We will also need the following analogue of Lemma~\ref{lem:out-dom-trans} for in-dominating sets. It immediately follows from
Lemma~\ref{lem:out-dom-trans} by reversing the orientations of all edges.

\begin{lem}
\label{lem:in-dom-trans}Let $T$ be a tournament on $n$ vertices,
let $v\in V(T)$, and suppose that $c\in\mathbb{N}$ satisfies $2\le c\le\log d^{+}(v)-1$.
Then there exist disjoint sets $B,E\subseteq V(T)$ such that the
following properties hold:
\begin{enumerate}
\item $2\le |B|\le c$ and $T[B]$ is a transitive tournament with tail $v$.
\item $B$ in-dominates $V(T)\setminus (B\cup E)$.
\item $|E|\le(1/2)^{c-1}d^{+}(v)$.
\end{enumerate}
\end{lem}

We will now apply Lemma~\ref{lem:out-dom-trans} repeatedly to obtain many pairwise disjoint small
almost-out-dominating sets. We will also prove an analogue for in-dominating sets.
These lemmas will be used in order to obtain sets $A_{1}^{1},\dots,A_{k}^{t}$, $B_{1}^{1},\dots,B_{k}^{t}$,
$E_{A,1},\dots,E_{A,k}$ and $E_{B,1},\dots,E_{B,k}$ as in
Definition~\ref{def:good}.

\begin{lem}
\label{lem:findA*s}Let $T$ be a tournament on $n$ vertices, 
$U\subseteq V(T)$ and $c\in \mathbb{N}$ with $c\ge 2$. Suppose that $\delta^{-}(T)\ge2^{c+1}+c|U|$.
Then there exist families $\{A_{v}:v\in U\}$ and $\{E_{v}:v\in U\}$
of subsets of $V(T)$ such that the following properties hold:
\begin{enumerate}
\item $A_{v}$ out-dominates $V(T)\setminus (E_{v}\cup \bigcup_{u\in U}A_{u})$
for all $v\in U$.
\item $T[A_{v}]$ is a transitive tournament with head $v$ for all $v\in U$.
\item $|E_{v}|\le(1/2)^{c-1}d^{-}(v)$ for all $v\in U$.
\item $2\le |A_{v}|\le c$ for all $v\in U$.
\item $A_{u}\cap E_{v}=\emptyset$ for all $u,v\in U$.
\item $A_{u}\cap A_{v}=\emptyset$ for all $u\ne v$.
\end{enumerate}
\end{lem}
\begin{proof}
We repeatedly apply Lemma~\ref{lem:out-dom-trans}. Suppose that for some $U'\subseteq U$ with $U'\neq U$ we have
already found $\{A_{u}:u\in U'\}$ and $\{E'_{u}:u\in U'\}$ satisfying (ii)--(vi) (with $U'$ playing the role of $U$ and $E_u'$ playing the role of $E_u$) such that
\begin{itemize}
\item[(a)] $A_{v}$ out-dominates $V(T)\setminus (\bigcup_{u\in U'}A_{u}\cup E'_{v}\cup U)$
for all $v\in U'$;
\item[(b)] $(\bigcup_{u\in U'}A_{u})\cap U=U'$.
\end{itemize}
Pick $v\in U\setminus U'$. Our aim is to apply Lemma~\ref{lem:out-dom-trans}
to $v$ and
\[
T':=T-\left(\bigcup_{u\in U'}A_{u}\cup(U\setminus\{v\})\right).
\]
 Note that $v\in V(T')$ by (b). Moreover,
\begin{align*}
d^-_{T'}(v) \ge\delta^{-}(T') \stackrel{{\rm (iv)}}{\ge}
\delta^{-}(T)-c|U'|-|U\setminus U'| \ge\delta^{-}(T)-c|U| \ge2^{c+1},
\end{align*}
where the final inequality holds by hypothesis, and so $c\le\log d_{T'}^{-}(v)-1.$
Hence 
we can apply Lemma~\ref{lem:out-dom-trans} to
obtain disjoint sets $A_{v},E_{v}\subseteq V(T')$ as described there. For all $u\in U'$, let $E_{u}:=E'_{u}\setminus A_{v}$.
Then the collections $\{A_{u}:u\in U'\cup \{v\}\}$ and $\{E_{u}:u\in U'\cup \{v\}\}$ satisfy
(v) and (vi) (with $U'\cup \{v\}$ playing the role of $U$). Moreover, (b) holds too (with $U'\cup \{v\}$ playing the role of $U'$).
Conditions (i)--(iii) of Lemma~\ref{lem:out-dom-trans} imply that (a) holds (with $U'\cup \{v\}$, $E_{u}$ playing the roles of $U'$, $E'_u$)%
   \COMMENT{Note that for each $v'\in U'$ we have $\bigcup_{u\in U'}A_{u}\cup E'_{v'}\cup U \subseteq \bigcup_{u\in U'\cup \{v\}}A_{u}\cup E_{v'}\cup U$.
So $A_{v'}$ out-dominates $V(T)\setminus (\bigcup_{u\in U'\cup \{v\}}A_{u}\cup E_{v'}\cup U)$.}
and that (ii)--(iv) hold (with $U'\cup \{v\}$ playing the role of $U$).

We continue in this way to obtain sets $\{A_{u}:u\in U\}$ and $\{E_{u}:u\in U\}$ which satisfy (ii)--(vi) as well as (a)
(with $U$, $E_{u}$ playing the roles of $U'$, $E'_{u}$). But (a) implies (i) since $\bigcup_{u\in U}A_{u}\cup U=\bigcup_{u\in U}A_{u}$ (as $u\in A_u$ by (ii)).
\end{proof}

The next lemma is an analogue of Lemma~\ref{lem:findA*s} for in-dominating sets. The proof is similar to that of Lemma~\ref{lem:findA*s}.

\begin{lem}
\label{lem:findB*s}Let $T$ be a tournament on $n$ vertices,
$U\subseteq V(T)$ and $c\in \mathbb{N}$ with $c\ge 2$. Suppose that $\delta^{+}(T)\ge2^{c+1}+c|U|$.
Then there exist families $\{B_{v}:v\in U\}$ and $\{E_{v}:v\in U\}$
of subsets of $V(T)$ such that the following properties hold:
\begin{enumerate}
\item $B_{v}$ in-dominates $V(T)\setminus (E_v \cup \bigcup_{u\in U}B_{u})$ for all $v\in U$.
\item $T[B_{v}]$ is a transitive tournament with tail $v$ for all $v\in U$.
\item $|E_{v}|\le(1/2)^{c-1}d^{+}(v)$ for all $v\in U$.
\item $2\le |B_{v}|\le c$ for all $v\in U$.
\item $B_{u}\cap E_{v}=\emptyset$ for all $u,v\in U$.
\item $B_{u}\cap B_{v}=\emptyset$ for all $u\ne v$.
\end{enumerate}
\end{lem}

We will now combine the previous results in order to prove that any sufficiently highly-linked tournament
is $(C,k,t,c)$-good. Note that Lemmas~\ref{lem:good-embed} and~\ref{lem:good} together imply
Theorem~\ref{thm:mainresult}.

\begin{lem}\label{lem:good}
Let $C:=10^{7}$, $k\ge20$, $t:=164k$ and $c:=\lceil\log 50t+1\rceil$. Then
any $Ck^{2}\log k$-linked tournament is $(C,k,t,c)$-good.
\end{lem}
\begin{proof}
Let $T$ be a $Ck^{2}\log k$-linked tournament, and let $n:=|T|$.
Note in particular that $\delta^{0}(T)\ge Ck^{2}\log k$ by Proposition~\ref{prop:samelinkage},
so (G8) is satisfied. We have to choose $A_{1}^{1},\dots,A_{k}^{t}$, $B_{1}^{1},\dots,B_{k}^{t}$,
$E_{A,1},\dots,E_{A,k}$, $E_{B,1},\dots,E_{B,k}$, $F_{1},\dots,F_{k}$ and
$P_{1}^{1},\dots,P_{k}^{t}$ satisfying (G1)--(G7) of Definition~\ref{def:good}.

Construct a set $A\subseteq V(T)$ by greedily choosing $kt$ vertices of least possible in-degree in $T$, and likewise 
construct a set $B'\subseteq V(T)$ by greedily choosing $kt$ vertices of least possible out-degree in $T$.\COMMENT{"Minimal 
in-degree" sounded as though we could guarantee $kt$ vertices $v$ with $d^-(v) = \delta^-(T)$. -JL}
Note that by choosing the vertices
in~$A$ and~$B'$ suitably, we may assume that $A\cap B'=\emptyset$. (Since $n\ge\delta^{0}(T)\ge2kt$,
this is indeed possible.) Define
\begin{align*}
d_{-} & :=\min\{d^{-}(v):v\in V(T)\setminus(A\cup B')\},\\
d_{+} & :=\min\{d^{+}(v):v\in V(T)\setminus(A\cup B')\}.
\end{align*}
Note that  $d^-(a) \le d_-$ for all $a \in A$ and $d^+(b) \le d_+$ for all $b\in B'$.

Our first aim is to choose the sets $A_{1}^{1},\dots,A_{k}^{t}$ using Lemma~\ref{lem:findA*s}.
Partition $A$ arbitrarily into sets $A_{1},\dots,A_{k}$ of size
$t$, and write $A_{i}=:\{a_{i}^{1},\dots,a_{i}^{t}\}$. Since $|B'|=kt\le \delta^0(T)/2$, we have
\[
2^{c+1}+c|A|\le 400t+ckt\le \frac{C}{2}k^{2}\log k\le\delta^{-}(T)-|B'|\le \delta^-(T-B').
\]
Thus we can apply Lemma~\ref{lem:findA*s} to $T-B'$, $A$ and $c$ in order 
to obtain almost out-dominating sets $A_{i}^{\ell}\ni a_{i}^{\ell}$
and corresponding exceptional sets $E_{A,i}^{\ell}$ as in the statement
of Lemma~\ref{lem:findA*s} (for all $i\in [k]$ and all $\ell \in [t]$).
Write $A_{i}^{*}:=A_{i}^{1}\cup\dots\cup A_{i}^{t}$
and $A^{*}:=A_{1}^{*}\cup\dots\cup A_{k}^{*}$. 

Let us now verify~(G1). By Lemma~\ref{lem:findA*s}(ii), (iv) and~(vi), each
$T[A_{i}^{\ell}]$ is a transitive tournament with head $a_i^\ell$,
$2\le |A_{i}^{\ell}|\le c$, and the sets $A_{1}^{1},\dots,A_{k}^{t}$ are all disjoint.
In particular, $A=\{h(A_{i}^{\ell}):i\in[k],\ell\in[t]\}$.
We claim in addition that $d^{+}(a_{i}^{\ell})\ge2n/5$. Indeed, Proposition~\ref{prop:degseqbound}
implies that $T$ has at most $4n/5+1$ vertices of out-degree at most $2n/5$, and hence at least $n/5-1$
vertices of out-degree at least $2n/5$. Moreover,
\[
|A|=kt\le\frac{Ck^{2}\log k}{5}-1\le\frac{n}{5}-1.
\]
So since the vertices of $A$ were chosen to have minimal in-degree
in $T$, it follows that $d^{+}(a_{i}^{\ell})\ge2n/5$ for all $i\in [k]$ and all $\ell\in [t]$. Thus~(G1) holds. 

We will next apply Lemma~\ref{lem:findB*s} in order to obtain the sets $B_{1}^{1},\dots,B_{k}^{t}$.
To do this, we first partition $B'$ arbitrarily into
sets $B^\prime_{1},\dots,B^\prime_{k}$ of size $t$, and write $B^\prime_{i}=:\{b_{i}^{\prime1},\dots,b_{i}^{\prime t}\}$.
Since
$|A^{*}|\le ktc\le\delta^{0}(T)/2$, we have
\[
2^{c+1}+c|B|\le 400t+ckt\le\frac{C}{2}k^{2}\log k\le\delta^+(T)-|A^{*}|\le\delta^+(T-A^{*}).
\]
Thus we can apply Lemma~\ref{lem:findB*s} to $T-A^{*}$, $B'$ and $c$ in order
to obtain almost in-dominating sets $B_{i}^{\ell}\ni b_{i}^{\prime \ell}$
and corresponding exceptional sets $E_{B,i}^{\ell}$ as in the statement
of Lemma~\ref{lem:findB*s} (for all $i\in [k]$ and all $\ell \in [t]$).
Write $B_{i}^{*}:=B_{i}^{1}\cup\dots\cup B_{i}^{t}$
and $B^{*}:=B_{1}^{*}\cup\dots\cup B_{k}^{*}$. Similarly as before one can show
that~(G2) holds.%
    \COMMENT{First note that $A_{1}^{1},\dots,A_{k}^{t},B_{1}^{1},\dots,B_{k}^{t}$
are all disjoint. By Lemma~\ref{lem:findB*s}(ii) and (iv)  each
$T[B_{i}^{\ell}]$ is a transitive tournament with tail $b_i^{\prime\ell}$, and
$2\le |B_{i}^{\ell}|\le c$. In particular, $B'=\{t(B_{i}^{\ell}):i\in[k],\ell\in[t]\}$.
We claim in addition that $d^{-}(b_{i}^{\prime \ell})\ge2n/5$. Indeed, Proposition~\ref{prop:degseqbound}
implies that $T$ has at most $4n/5+1$ vertices of in-degree at most $2n/5$, and hence
at least $n/5-1$ vertices of in-degree at least $2n/5$. Moreover
\[
|B'|=kt\le\frac{Ck^{2}\log k}{5}-1\le\frac{n}{5}-1.
\]
So since the vertices of $B'$ were chosen to have minimal out-degree
in $T$ we have $d^{-}(b_{i}^{\prime \ell})\ge2n/5$ for all $i\in [k]$ and all $\ell\in [t]$.
Thus (G2) holds.}

We now define the exceptional sets $E_{A,i}$ and $E_{B,i}$.
For all $i\in [k]$, let 
$$
E_{A,i}:=(E_{A,i}^{1}\cup\dots\cup E_{A,i}^{t})\setminus B^{*} \ \ \mbox{ and } \ \
E_{B,i}:=(E_{B,i}^{1}\cup\dots\cup E_{B,i}^{t}).
$$ 
Recall
from Lemmas~\ref{lem:findA*s}(v) and~\ref{lem:findB*s}(v)
that $E_{A,i}^{\ell}\cap A^{*}=\emptyset$ and $E_{B,i}^{\ell}\cap(A^{*}\cup B^{*})=\emptyset$
for all $i\in [k]$ and all $\ell \in [t]$. Thus $E_{A,i}\cap(A_{i}^{*}\cup B_{i}^{*})=\emptyset$
and $E_{B,i}\cap(A_{i}^{*}\cup B_{i}^{*})=\emptyset$ for all $i\in [k]$.
By Lemma~\ref{lem:findA*s}(i), each $A_{i}^{\ell}$ out-dominates $V(T)\setminus (A^{*}\cup B^{*}\cup E_{A,i})$.
Lemma~\ref{lem:findA*s}(iii) and the fact that $a_i^\ell \in A$ together imply that
\begin{equation}\label{eq:sizeEi+}
|E_{A,i}|\le\sum_{\ell=1}^{t}|E_{A,i}^{\ell}|\le\sum_{\ell=1}^{t}\frac{1}{2^{c-1}}d^{-}(a_{i}^{\ell})\le\frac{t}{2^{c-1}}d_{-}\le \frac{d_{-}}{50},
\end{equation}
so (G3) holds. Similarly, by Lemma~\ref{lem:findB*s}(i), each $B_{i}^{\ell}$
in-dominates $V(T)\setminus (A^{*}\cup B^{*}\cup E_{B,i})$,
and as in (\ref{eq:sizeEi+}) one can show that $|E_{B,i}|\le d_{+}/50$. Thus (G4) holds. 

We now use Lemma~\ref{lem:coveringedge} in order to define the sets $F_{1},\dots,F_{k}$ of covering edges.
Recall from (G7) that we require $F_{1}\cup\dots\cup F_{k}$ to be a matching in $T-(A^{*}\cup B^{*})$.
Suppose that for some (possibly empty) subset $V' \subsetneq A^*\cup B^*$ we have defined a set $\{e_v:v\in V'\}$ of independent edges 
in $T-(A^* \cup B^*)$ such that 
$e_v$ is a covering edge for $v$ and $e_v \ne e_{v'}$ whenever $v \ne v'$.
Pick any vertex $v\in (A^*\cup B^*)\setminus V'$. We will next define $e_v$.
Let $T'$ be the tournament obtained from $T$ by deleting $(A^* \cup B^*)\setminus \{v\}$ as well as the endvertices of
the covering edges $e_{v'}$ for all $v'\in V'$.
Then
\[
|V(T)\setminus V(T')|\le |A^{*}\cup B^{*}|+2|A^{*}\cup B^{*}|\le 3ktc\le\frac{C}{2}k^{2}\log k,
\]
so by Proposition~\ref{prop:robustlinkage}, $T'$ is still $(Ck^{2}\log k/2)$-linked and hence strongly 2-connected.
We may therefore apply Lemma~\ref{lem:coveringedge} to find a covering
edge $e_{v}$ for $v$ in $T'$.  Continue in this way until we have chosen $e_v$ for each $v\in A^*\cup B^*$
and let $F_{i}:=\{e_{v}:v\in A_{i}^{*}\cup B_{i}^{*}\}$. Then the first part of (G7) holds.

It remains to choose the paths $P_{1}^{1},\dots,P_{k}^{t}$. Recall from (G6) that we need
to ensure that $(A\cup B')\setminus(A_{i}^{*}\cup B_{i}^{*})\subseteq V(P_{i}^{t})$
for all $i\in [k]$. We could achieve this by incorporating each of these vertices using
the high linkedness of~$T$. However, since $|A\cup B'|=2kt$, a direct application of linkedness
would require $T$ to be $\Theta(k^{3})$-linked. For each $i\in [k]$, we will therefore
first choose a path cover $\mathcal{Q}_{i}$ of $T[(A\cup B')\setminus(A_{i}^{*}\cup B_{i}^{*})]$
consisting of few paths and then use Lemma~\ref{lem:shortlinkagewithpaths}
(and thereby the high linkedness of $T$) to incorporate these paths
into~$P_{i}^{t}$. This has the advantage that we will only need $T$ to be $\Theta(k^{2}\log k)$-linked. 

Let us first choose the path covers $\mathcal{Q}_{i}$ of $T[(A\cup B')\setminus(A_{i}^{*}\cup B_{i}^{*})]$.
Suppose that for some $j\in [k]$ we have already found path systems $\mathcal{Q}_{1},\dots,\mathcal{Q}_{j-1}$
such that, for each $i< j$, $\mathcal{Q}_{i}$ is a path cover of $T[(A\cup B')\setminus(A_{i}^{*}\cup B_{i}^{*})]$ with
$|\mathcal{Q}_{i}|\le 2k$, and such that for all $i<i'<j$ the paths in $\mathcal{Q}_{i}$ are edge-disjoint from paths
in $\mathcal{Q}_{i'}$. To choose $\mathcal{Q}_{j}$, apply Corollary~\ref{cor:pathcover}
to the oriented graph $T''$ obtained from $T[(A\cup B')\setminus(A_{j}^{*}\cup B_{j}^{*})]$ by deleting the edges
of all the paths in $\mathcal{Q}_{1},\dots,\mathcal{Q}_{j-1}$. Since $\delta(T'')\ge |T''| -1 -2(j-1)\ge |T''|-2k$,
Corollary~\ref{cor:pathcover} ensures that $|\mathcal{Q}_{j}|\le 2k$.

We will now choose $P_{1}^{1},\dots,P_{k}^{t}$. For each $i\in [k]$ and each $\ell\in [t]$, let $a_{i}^{\prime \ell}$
denote the tail of $T[A_{i}^{\ell}]$ and $b_{i}^{\ell}$ the head of $T[B_{i}^{\ell}]$.
Let 
$$
A':=\{a_{i}^{\prime \ell}:i\in [k], \ \ell\in [t]\} \ \ \mbox{ and } \ \  B:=\{b_{i}^{\ell}:i\in [k], \ \ell\in [t]\}.
$$
For all $i\in [k]$ and all $\ell\in [t-1]$ let $\mathcal{Q}_i^\ell:=\emptyset$. For all $i\in [k]$
let $\mathcal{Q}_i^t$ be the path system consisting of all the edges in $F_i$ (each viewed as a path of length one)
and all the paths in $\mathcal{Q}_i$.%
   \COMMENT{Since $F_i\subseteq T-(A^*\cup B^*)$, this is indeed a path system.}
Let $T''':=T-((A^*\cup B^*)\setminus (A\cup A'\cup B\cup B'))$.
Our aim is to apply Lemma~\ref{lem:shortlinkagewithpaths} with $s:=30$ to $T'''$,
the vertices $b_{1}^{1}, \dots, b_k^t$, $a_{1}^{\prime 1},\dots,a_{k}^{\prime t}$,
and the path systems $\mathcal{Q}_1^1, \dots, \mathcal{Q}_k^t$. To verify that $T'''$ is sufficiently highly linked,
let $m$ be as defined in~(\ref{eq:defm}) and note that
\begin{align*}
m & =kt+3\sum_{i=1}^{k}|F_{i}|+\sum_{i=1}^{k}|\mathcal{Q}_{i}|+\left|\bigcup_{i=1}^{k}V(\mathcal{Q}_{i})\right|
 \le kt+6ckt+2k^{2}+|A\cup B'|\\ & \le 5kt+6ckt\le\frac{C}{70}k^{2}\log k.
\end{align*}
Together with the fact that $|T|-|T'''|\le 2ckt$ and Proposition~\ref{prop:robustlinkage} this implies that
$T'''$ is $2\cdot 30 m$-linked. So we can indeed apply Lemma~\ref{lem:shortlinkagewithpaths}
to find edge-disjoint paths $P_i^\ell$ in $T'''$ (for all $i\in [k]$ and all $\ell\in [t]$) satisfying the following properties:
\begin{enumerate}
\item $P_{i}^\ell$ is a path from $b_{i}^{\ell}$ to $a_{i}^{\prime \ell}$.
\item $Q\subseteq P_{i}^\ell$ for all $Q\in\mathcal{Q}_{i}^\ell$. 
\item $V(P_{i}^\ell)\cap V(P_{j}^{m})\subseteq V(\mathcal{Q}_{i}^\ell)\cap V(\mathcal{Q}_{j}^{m})$ for all $(i,\ell)\neq (j,m)$.
\item We have that
\begin{align*}
|P_{1}^{1}\cup\dots\cup P_{k}^{t}| & \le\frac{n}{30}+ 2\sum_{i=1}^{k}|F_{i}|+ \left|\bigcup_{i=1}^{k}V(\mathcal{Q}_{i})\right|
= \frac{n}{30}+2|A^*\cup B^*|+|A\cup B'|\\
& \le \frac{n}{30}+4ckt+2kt\le\frac{n}{20}.
\end{align*}
\end{enumerate}
Condition (ii) implies that $F_i\subseteq P_i^t$ and
$(A\cup B')\setminus(A_{i}^{*}\cup B_{i}^{*})=V(\mathcal{Q}_{i})\subseteq V(\mathcal{Q}_{i}^t)\subseteq V(P_{i}^{t})$ for all $i\in [k]$.
Thus (G6) holds.

We now prove that (G5) holds.
From (iii) and the fact that that $V(\mathcal{Q}_i^\ell) \cap V(\mathcal{Q}_i^m) = \emptyset$ for all $i\in [k]$, $\ell\ne m$, it follows that 
$P_i^1,\dots,P_i^t$ are vertex-disjoint for all $i\in [k]$.
Together with (i) and (iv) this implies that in order to check (G5), it remains to show that
\begin{equation}\label{eq:nearlythere2}
V(\textrm{Int}(P_i^\ell)) \cap  (A^*\cup B^*) \subseteq (A\cup B')\setminus (A_i^*\cup B_i^*)\ \ \ \ \textrm{ for all }i\in [k],\ell\in [t].
\end{equation}
Clearly,
\begin{align}\label{eq:VPil2}
V(P_i^\ell)\cap (A^*\cup B^*)&\subseteq V(T''')\cap (A^*\cup B^*)\\
                             & = A\cup A'\cup B\cup B' \ \ \ \ \textrm{ for all }i\in [k],\ell\in [t].\nonumber
\end{align}
By definition, we have $(A'\cup B)\cap V(\mathcal{Q}_j^m) = \emptyset$ 
for all $j\in [k],m\in [t]$. It therefore follows from (iii) that each vertex in $A'\cup B$ may appear in at most
one path $P_j^m$. However, by~(i) each vertex in $A'\cup B$ is an endpoint of $P_j^m$ for some $j\in [k],m\in [t]$. 
Hence
\begin{equation}\label{eq:blabla}
V(\textrm{Int}(P_i^\ell)) \cap (A'\cup B) = \emptyset \ \ \ \ \textrm{ for all }i\in [k],\ell\in [t].
\end{equation}
Fix $i\in [k], \ell\in [t]$ and take $j\in [k]\setminus\{i\}$. We have $(A \cup B') \cap (A_i^* \cup B_i^*) \cap V(\mathcal{Q}_i^\ell) = \emptyset$, 
and by (G6) we have $(A\cup B')\cap (A_i^* \cup B_i^*) \subseteq (A\cup B')\setminus (A_j^* \cup B_j^*) \subseteq V(P_j^t)$.
Applying (iii) to $P_i^\ell$ and $P_j^t$, it therefore follows that
\begin{equation}\label{eq:lastone}
V(P_i^\ell) \cap (A \cup B') \cap (A_i^* \cup B_i^*) = \emptyset \ \ \ \ \textrm{ for all }i\in [k], \ell\in [t].
\end{equation}
(\ref{eq:VPil2})--(\ref{eq:lastone}) now imply (\ref{eq:nearlythere2}). Thus (G5) holds.

So it remains to check that the last part of (G7) holds too, i.e.~that
$F^{\rm act}_{i}\cap E(P_{j}^{\ell})=\emptyset$ for all $i,j\in [k]$ and all $\ell\in [t]$.
Consider any covering edge $e_{v}=x_{v}y_{v}\in F_{i}$. Then (G6) implies that $x_{v}$ and $y_{v}$
are contained in $P_{i}^{t}$. Moreover, (iii) implies that
$V(P_{i}^t)\cap V(P_{j}^{\ell})\subseteq V(\mathcal{Q}_{i}^t)\cap V(\mathcal{Q}_{j}^{\ell})\subseteq A\cup B'$
whenever $(i,t)\neq (j,\ell)$. Since $x_{v}, y_{v}\notin A\cup B'$, this shows that $x_{v}v,vy_{v}\notin E(P_{j}^{\ell})$
whenever $(i,t)\neq (j,\ell)$. But since $e_{v}\in E(P_{i}^{t})$,
we also have $x_{v}v,vy_{v}\notin E(P_{i}^{t})$. This completes the proof that $T$ is $(C,k,t,c)$-good.
\end{proof}

\section{Concluding remarks} \label{concluding}

\subsection{Eliminating the logarithmic factor}
A natural approach to improve the bound in Theorem~\ref{thm:mainresult} would be to reduce the parameter $c$, 
i.e.~to consider smaller `almost dominating' sets. In particular, if we could choose $c$ independent of $k$, then 
we would obtain the (conjectured) optimal bound of $\Theta(k^2)$ for the linkedness.
The obstacle to this in our argument is given by~(\ref{eq:sizeEi+}),
which requires that $c$ has a logarithmic dependence on~$k$.%

\subsection{Algorithmic aspects}
As remarked in the introduction, the proof of Theorem~\ref{thm:mainresult} is algorithmic.
Indeed, when we apply the assumption of high linkedness to find appropriate paths in the proof of Lemma~\ref{lem:good}
(via Lemma~\ref{lem:shortlinkagewithpaths}), we can make use of the main result of~\cite{CSS} that these can be found in polynomial time.
Moreover, the proof of the Gallai-Milgram theorem (Theorem~\ref{thm:GM}) is also algorithmic (see~\cite{BondyMurty}).
These are the only tools we need in the proof, and the proof itself immediately translates into a polynomial time algorithm.%
\COMMENT{full algorithm:
\begin{enumerate}
\item Construct vertex sets $A,B'\subseteq V(T)$ by greedily taking $kt$ vertices of least possible in- and out-degrees as in the proof of Lemma~\ref{lem:good}. 
\item Extend $A$ and $B'$ into dominating sets $A_1^1, \dots, A_k^t\subseteq V(T)$ and $B_1^1, \dots, B_k^t\subseteq V(T)$ with respective exceptional sets $E_{A,1},\dots, E_{A,k}$ and $E_{B,1},\dots, E_{B,k}$ using the proofs of Lemmas~\ref{lem:findA*s} and \ref{lem:findB*s}. Define $A_i^* := A_i^1 \cup \dots \cup A_i^t$ and $B_i^* := B_i^1 \cup \dots \cup B_i^t$ for all $i \in [k]$.
\item Find sets $F_1, \dots, F_k\subseteq E(T)$ of covering edges by brute force.\COMMENT{This seems simpler than arguing for the process being polynomial.}
\item Use Theorem~\ref{thm:GM} to find path covers $\mathcal{Q}_1, \dots, \mathcal{Q}_k$ as in the proof of Lemma~\ref{lem:good-embed}.
\item Apply linkage to incorporate $\mathcal{Q}_1, \dots, \mathcal{Q}_k$ into paths $P_1^1, \dots, P_k^t\subseteq T$ as in the proof of Lemma~\ref{def:good}. 
\item Set $i:=1$, $C_1, \dots, C_k:=\emptyset$.
\item If $i>k$, halt and output $C_1,\dots,C_k$. Otherwise, define $T_i$, $E_{A,i}'$, $E_{B,i}'$ and $X_i$ as in the proof of Lemma~\ref{lem:good-embed}.
\item Form a path system $\mathcal{Q}$ in $T_i$ with properties (Q1)--(Q7) with respect to $E_{A,i}'$ and $E_{B,i}'$ by applying Theorem~\ref{thm:GM} and the proofs of Lemmas~\ref{lem:pathextend-out} and \ref{lem:pathextend-in} as in the proof of Lemma~\ref{lem:mainengine}.
\item Link these paths into a cycle $C_i'$ covering $T_i-(A_i^* \cup B_i^*)$ using the domination of $A_i^*$ and $B_i^*$ as in the proof of Lemma~\ref{lem:mainengine}. (Find the edges out of $A_i^*$ and the edges into $B_i^*$ by brute force.)
\item Turn $C_i'$ into a Hamilton cycle $C_i$ by deploying covering edges in $F_i$ as appropriate.
\item Increment $i$ and go to step (vii).
\end{enumerate}
It is therefore immediate that steps (i), (iii)--(vii) and (ix)--(xi) take only polynomial time, and that no step is repeated more than $k$ times. The proofs of  Lemmas~\ref{lem:findA*s} and \ref{lem:findB*s} consist of polynomially many applications of Lemmas~\ref{lem:out-dom-trans} and \ref{lem:in-dom-trans} respectively. The proofs of Lemmas~\ref{lem:out-dom-trans} and \ref{lem:in-dom-trans} immediately yield polynomial time algorithms, so step (ii) is valid and takes polynomial time. The proofs of Lemmas~\ref{lem:pathextend-in} and \ref{lem:pathextend-out} consist of finding edges incident between the endpoints of paths and given sets, which can be done in polynomial time by brute force, so step (viii) is valid and takes polynomial time. Hence the entire algorithm takes polynomial time.}
\COMMENT{Note to self: remember to make sure everything in the bibliography is still being used when we reach the final version. -JL}

\medskip

{\footnotesize \obeylines \parindent=0pt

Daniela K\"{u}hn, John Lapinskas, Deryk Osthus, Viresh Patel 
School of Mathematics 
University of Birmingham
Edgbaston
Birmingham
B15 2TT
UK
\begin{flushleft}
{\it{E-mail addresses}:\\
\tt{\{d.kuhn, jal129, d.osthus, v.patel.3\}@bham.ac.uk}}
\end{flushleft}}

\end{document}